\documentclass[a4paper,fleqn]{article}
\usepackage{amsmath}
\usepackage{latexsym}
\usepackage{amssymb}
\usepackage{amsfonts}
\usepackage{mathrsfs}
\usepackage[dvips]{graphics}
\usepackage[all]{xy}
\usepackage{amsthm}
\usepackage{enumerate}
\usepackage{url}
\usepackage[yyyymmdd]{datetime}

\DeclareMathOperator{\dom}{\mathrm{dom}}
\DeclareMathOperator{\rng}{\mathrm{rng}}

\DeclareMathOperator{\Fin}{\mathit{Fin}}
\DeclareMathOperator{\Sep}{\mathit{Sep}}

\newcommand{\dpls}{\, \udot{+}\, }
\newcommand{\upls}{\, \dot{+}\, }
\newcommand{\dmns}{\, \udot{-}\, }
\newcommand{\umns}{\, \dot{-}\, }
\newcommand{\dtms}{\, \udot{\times}\, }
\newcommand{\utms}{\, \dot{\times}\, }
\newcommand{\ddiv}{\, \udot{\div}\, }
\newcommand{\udiv}{\, \dot{\div}\, }

\newcommand{\isrc}[1]{{\textstyle Q\hspace{0cm}\langle #1\rangle}}
\newcommand{\choice}[2]{{\mathrm{Ch}_{{#1}}\hspace{0cm}(#2)}}
\newcommand{\udot}[1]{\text{\d{$\displaystyle #1$}}}
\newcommand{\intrr}[2]{\mathrm{int}_{\mathscr{#1}}\left({#2}\right)}
\newcommand{\fig}[2]{\mathrm{fig}_{\mathscr{#1}}\left(#2\right)}
\newcommand{\cl}[2]{\mathrm{cl}_{\mathscr{#1}}\left(#2\right)}

\newcommand{\mon}[2]{\mathrm{mon}_{\mathscr{#1}}\left(#2\right)}

\newcommand*{\FN}{\mathit{FN}}
\newcommand*{\FQ}{\mathit{FQ}}
\newcommand*{\BQ}{\mathit{BQ}}

\let\uhr\upharpoonright
\renewcommand*{\upharpoonright}{\hspace{-.07cm}\uhr\hspace{-.07cm}}
\newtheorem{thm}{Theorem}
\newtheorem{prop}[thm]{Proposition}
\newtheorem{cor}[thm]{Corollary}
\newtheorem{lemma}[thm]{Lemma}
\theoremstyle{remark}

\newtheorem{remark}{Remark}
\newtheorem{ex}{Example}
\theoremstyle{definition}

\newtheoremstyle{axiom1}
{3pt}
{3pt}
{\itshape}
{}
{\bfseries\itshape}
{}
{.5em}
{}
\newtheoremstyle{axiom2}
{3pt}
{3pt}
{\itshape}
{}
{\bfseries\itshape}
{\\}
{.5em}
{}
\theoremstyle{axiom1}
\title{A Foundation of $\sigma$-superadditive Measures\\ -- a note on advancing Kalina measures --}
\author{Kiri Sakahara\thanks{Yokohama National University and Kanagawa University, Kanagawa, Japan.} \and Takashi Sato\thanks{Toyo University, Tokyo, Japan.}}
\date{}
\begin{document}
\maketitle

\begin{abstract}
The present paper attempts to modify the way of constructing a measure in the Alternative Set Theory setting originally devised by Martin Kalina.
Introducing a system of cuts of rational numbers extended  with some special ones, it is proved that the measure which is nondecreasing, nonnegative and ``depending on the way of measurement'' as same as Kalina's, but exhibits $\sigma$-superadditivity is obtained.
\end{abstract}

\section{Introduction}
Among various attempts on developing theories on measures in Alternative Set Theory(AST, for short) \cite{ast}, the series of works \cite{Kalina1989-b,Kalina1989-a,Kalina-Zlatos1988,Kalina-Zlatos1989-b,Kalina-Zlatos1989-a} studied by Kalina and Zlato\v{s}, the first two of which are authored by Kalina alone, is of outstanding importance.
Not only it offers a consistent framework with that of traditional measure theory under the certain conditions, though Ra\v{s}kovi\'c \cite{Raskovic} also confirms independently that it is possible to construct measures which coincide with the framework of Loeb's, but also enables, especially Kalina \cite{Kalina1989-b,Kalina1989-a}, to measure classes on somewhat indefinite basis, say, the size of countable class $\FN$.

To put it more illustratively, let us consider stitching 1m long strings together again and again.
After repeating it countably many times, the whole string may become immeasurably long, and the size of it cannot be calculated specifically since how many of them are connected remains indefinite.
However, it is plausible to infer that it is half shorter than that of stitching 2m ones, so that the measure of the one stiching 1m long strings measured against that of 2m must be $\frac{1}{2}$.
Kalina \cite{Kalina1989-b,Kalina1989-a} organized a framework that enables to measure exactly that kind of length in a proper manner.

Despite its elaborate considerations, however, certain aspects are remained not fully investigated.
One example can be illustrated by the opposite problem.
Let us consider, conversely, cutting 1m long string in half repeatedly.
After countably many repetition, resulting pieces must be indiscernibly short, but they must also obey a certain type of arithmetic, too.
However, Kalina does not go further enough to answer this problem.

To complete this investigation in a convincing manner, the present paper introduces \textit{cuts of rationals} which  constitute an extended system of rational numbers.
They consist not only of cuts representing rational numbers but also some peculiar cuts, so called \textit{additive cuts} which satisfy the property $A+A=A$.
One of the most important ones is $\isrc{\frac{1}{\FN}}$ which represents an indiscernibly small number. 
It is constant no matter how many times it is added as long as it is finite. 
However, after adding up countably many times, the result jumps to any cut ranging from $\isrc{\frac{1}{\FN}}$ to $\isrc{\frac{\FN}{1}}$ depending on how it is approximated. 
This is the key property, said to be \textit{$\sigma$-superadditivity}, that enables not only to resolve the problem but also to measure classes which share properties with, say, \textit{Vitali sets} or the sets resulted from the theorem of Banach-Tarski \cite{Banach1924}. 

Moreover, the paper puts necessary definitions, theorems, and statements regarding Kalina's framework scattered across the papers together, and elaborates it to an extent to be able to measure specific classes more explicit manner.
Some examples are also provided.

\section{A System of Numbers}

Let us start with constructing a number system in accordance with Vop\v{e}nka \cite{ast}.
The class of \textit{natural numbers} $N$ is defined as:
\[
N\ =\ \left\{x\ ;\ 
\begin{matrix}
\left(\forall y\in x\right) \left(y\subseteq x\right)\\
\wedge\left(\forall y,z\in x \right) \left(y\in z \vee y=z \vee z\in y\right)
\end{matrix}
\right\},
\]
while the class of \textit{finite natural numbers} $\FN$ consists of the numbers represented by finite sets
\[
\FN \ =\ \left\{x\in N\ ;\ 
\Fin(x)\right\}
\]
in which $\Fin(x)$ means that each subclass of $x$ is a set.
The class of all integers $Z$ and that of all rational numbers are defined respectively as:
\[
Z\ =\ N\cup \left\{ -a ;\,
 a\in N\wedge a\ne 0
\right\}
\qquad\text{and}\qquad
Q\ =\ \left\{
\frac{x}{y}\ ;\ x,\,y\in Z \wedge y\ne 0
\right\}.
\]
$\BQ\subseteq Q$ denotes the class of \textit{bounded rational numbers} and $\FQ\subseteq BQ$ the class of \textit{finite rational numbers}, i.e., 
\begin{eqnarray*}
&& \BQ=\{x\in Q\ ;\ \left(\exists i\in\FN\right)
  \left(|x|\leq i\right)
\}
\; \text{and}\\
&& \FQ\ =\ \left\{
q \in Q\ ;\
\left(\exists x,y\in\FN\right)\left( q=\frac{x}{y}\vee q=-\frac{x}{y}\right)
\right\}.
\end{eqnarray*}

Real numbers are defined in AST as an equivalence class of rational numbers.
The reason behind this construction lies in the human's inability to distinguish two mutually close rational numbers.
This idea is grasped by the \textit{indiscernibility equivalence}, $\doteq$, on the class $Q$ of all rational numbers.
The definition of the indiscernibility equivalence $\doteq$ is given as: 
\[
p\doteq q\quad \equiv\quad
 \begin{pmatrix}
  \left( \exists k \right)
  \left( \forall i>0\right)
  \left(|p|<k \wedge |p-q|<\frac{1}{i}\right)\\[.18cm]
  \vee 
  \left(\forall k\right)
  \left(
  \left( p>k \wedge q>k\right)  \vee 
  \left(p<-k \wedge q<-k\right)
  \right)
 \end{pmatrix}
\]
in which the letters $i,j,k$ denote finite natural numbers, i.e., $i,j,k\in\FN$ for notational ease, hereafter.
For each $q\in Q$ the notation $\mon{}{q}=\{s\in Q\,;\, s\doteq q\}$, called the \textit{monad} of $q$, represents the class of all rational numbers which are indiscernible from $q$.
And every real number $a$ is represented by a monad $\mon{}{q}$ of some rational number $q$.

The class of all real numbers $R$ is defined as:
\[
R\ \equiv\ \left\{\mon{}{x}\ ;\ 
x\in \BQ
\right\}\ =\ \BQ/\doteq
\]
Two limit cases are denoted as:
\[
\infty\ =\ 
\{q\in Q\,;\, 
 \left(\forall i\right)
 \left(q>i\right)
\}\quad\text{ and }\quad
-\infty\ =\ 
\{q\in Q\,;\,
 \left(\forall i\right)
 \left(q<-i\right)
\}.
\]

Finite arithmetic operation of real numbers is same as usual.
However, the countable sum of real numbers cannot uniquely be determined as we shall see later in the paper.

\section{Borel Classes}\label{BC}

The system of all Borel classes, denoted as $\mathscr{B}$, is the smallest $\sigma$-ring\footnote{
The system of classes $\mathscr{M}$ is called a $\sigma$-ring if it is closed with respect to class-theoretical difference and countable unions, but  not required to contain a whole class, while a $\sigma$-field is.
} containing all the set theoretically definable classes, and $\mathscr{B}_s$ denotes that of all Borel semisets.
More precisely, Borel classes are defined inductively as
\begin{enumerate}[(1)]
  \setlength{\parskip}{0.072cm} 
  \setlength{\itemsep}{0.072cm} 
\item $X$ is a $\sigma_0$-class iff $X$ is a $\pi_0$-class iff $X$ is a set-theoretically definable class.
\item If $a\in\FN$, then $X$ is a $\sigma_a$-class ($\pi_a$-class) iff there is a sequence $(X_i)_{i\in\FN}$ such that each $X_i$ is a $\pi_{<a}$-class ($\sigma_{<a}$-class), which represents a $\pi_b$-class ($\sigma_b$-class) for some $b<a$, and $X=\cup_{i\in\FN} X_i$ ($\cap_{i\in\FN} X_i$).
\item $X$ is a $\delta_a$-class if $X$ is both a $\sigma_a$-class and $\pi_a$-class.
\item Finally, $X$ is a Borel class if $X$ is a $\sigma_a$-class or $\pi_a$-class for some $a\in\FN$.
\end{enumerate}

$\mathscr{B}$ is strictly larger than $V$ since it includes proper semisets, such as $\FN$, and, above all, $V$ itself.
The elements of $\mathscr{B}$ is called \textit{Borel classes}.
When $B\in\mathscr{B}$ is a semiset, it is called a \textit{Borel semiset}.
While $\mathscr{B}$ is a $\sigma$-field as well since it includes $V$ as a whole class, $\mathscr{B}_s$ is not.

Despite being defined as countable intersection/union of $\sigma_{<n}$/$\pi_{<n}$-classes for some $n\in\FN$, Borel classes can be built as $\delta_2$-class. 
The next theorem confirms it.
To describe such sequences conveniently, let $\cup\cap X_i$ and $\cap\cup X_i$ denote
\[
\bigcup_{i\in\FN}\bigcap_{j\geq i} X_j\quad\text{and}\quad
\bigcap_{i\in\FN}\bigcup_{j\geq i} X_j
\]
respectively, for a given sequence of classes $(X_i)_{i\in\FN}$.
\begin{thm}\label{pisigma}
For each Borel class $X\in\mathscr{B}$, there exists a sequence of set-theoretically definable classes $(X_i)_{i\in\FN}$ which satisfies $\cap\cup X_i=\cup\cap X_i=X$. 
\end{thm}

\begin{proof}
  Let $X\in\mathscr{B}$ be a $\pi_n$-class, where $n\geq 3$, and $(X(n-1)_i)_{i\in\FN}$ denotes a descending generating sequence of $X=X(n)$ for notational ease, consisting of $\sigma_{n-1}$-classes for simplicity.
  Let also $(X(n-2)_{i\ ;\ \ell_1})_{i\in\FN}$ denote an ascending generating sequence of $X(n-1)_{\ell_1}$ for each $\ell_1\in\FN$, consisting of $\pi_{n-2}$-classes.

  Similarly, let $\left(X(n-x)_{i\ ;\ \ell_1,\ldots,\ell_{x-1}}\right)_{i\in\FN}$ denote a descending generating sequence of $X(n-x+1)_{\ell_{x-1}\ ;\ \ell_1,\ldots,\ell_{x-2}}$ consisting of $\sigma_{n-x}$-classes if $x$ is odd, otherwise an ascending generating sequence consisting of $\pi_{n-x}$-classes.

  Define a sequence $\left(X_i\right)_{i\in\FN}$ of set-theoretically definable classes as $X_i=X(0)_{\ell_n\ ;\ i,\,\ell_2,\ldots,\ell_{n-1}}$ and $\ell_2,\ldots,\ell_n\in\FN$ are determined so as to satisfy the following inclusions for each $k\in\{0,\ldots, \lfloor n/2\rfloor-1\}$
  \begin{eqnarray*}
    X(n) & \subseteq & X(n-2k-2)_{\ell_{2k+2}\ ;\ i,\,\ell_2,\ldots,\ell_{2k+1}}\\
    &\subseteq& X(n-2k-3)_{\ell_{2k+3}\ ;\ i,\,\ell_2,\ldots,\ell_{2k+2}}\ \subseteq \ X(n-1)_i.
  \end{eqnarray*}
  Then, it is satisfied that
  \[
  X(n)\ \subseteq\ X_i\ \subseteq X(n-1)_i\quad\text{for all}\ i\in\FN.
  \]
  Since $X(n-1)_j\subseteq X(n-1)_i$ for all $j\geq i$, the following inclusions are met:
  \[
  X(n)\ \subseteq \ \cup_{j\geq i} X_{j}\ \subseteq \ X(n-1)_{i}\quad\text{for all}\ i\in\FN.
  \]
  Consequently, the equation $\cap\cup X_i=X(n)$ holds.

  By the same argument, the following inclusions are also satisfied:
  \[
  X(n)\ \subseteq \ \cap_{j\geq i} X_{j}\ \subseteq \ \cap_{j\geq i}X(n-1)_{j}\quad\text{for all}\ i\in\FN.
  \]
  Since $\cap_{j\geq i}X(n-1)_j=X(n)$ for all $i\in\FN$, the equation $\cup\cap X_i=X(n)$ holds.
  
  Next, let $X\in\mathscr{B}$ be a $\sigma_n$-class, where $n\geq 3$, and $\left(X(n-x)_{i\ ;\ \ell_1,\ldots,\ell_{x-1}}\right)_{i\in\FN}$ denote an ascending generating sequence of $X(n-x+1)_{\ell_{x-1}\ ;\ \ell_1,\ldots,\ell_{x-2}}$ consisting of $\pi_{n-x}$-classes if $x$ is odd, otherwise a descending generating sequence consisting of $\sigma_{n-x}$-classes.

  Define a sequence $\left(X_i\right)_{i\in\FN}$ of set-theoretically definable classes as $X_i=X(0)_{\ell_n\ ;\ i,\,\ell_2,\ldots,\ell_{n-1}}$ in which $\ell_2,\ldots,\ell_n\in\FN$ are determined so as to satisfy the following inclusions for each $k\in\{0,\ldots, \lfloor n/2\rfloor-1\}$
  \begin{eqnarray*}
    X(n-1)_{i} & \subseteq & X(n-2k-3)_{\ell_{2k+3}\ ;\ i,\,\ell_2,\ldots,\ell_{2k+2}}\\
    &\subseteq& X(n-2k-2)_{\ell_{2k+2}\ ;\ i,\,\ell_2,\ldots,\ell_{2k+1}}\ \subseteq \ X(n).
  \end{eqnarray*}
  Then, it is satisfied that
  \[
  X(n-1)_i\ \subseteq\ X_i\ \subseteq X(n)\quad\text{for all}\ i\in\FN.
  \]
  Since $X(n-1)_i\subseteq X(n-1)_j$ for all $i\leq j$, the following inclusions are met:
  \[
  X(n-1)_i\ \subseteq \ \cap_{j\geq i} X_{j}\ \subseteq \ X(n)\quad\text{for all}\ i\in\FN.
  \]
  Consequently, the equation $\cup\cap X_i=X(n)$ holds.

  By the same argument, the following inclusions are also satisfied:
  \[
  \cup_{j\geq  i}X(n-1)_{j}\ \subseteq \ \cup_{j\geq i} X_{j}\ \subseteq \ X(n)\quad\text{for all}\ i\in\FN.
  \]
  Since $\cup_{j\geq i}X(n-1)_j=X(n)$ for all $i\in \FN$, it holds that $\cap\cup X_i=X(n)$.
\end{proof}

The sequence $(X_i)_{i\in \FN}$ of set-theoretically definable classes which satisfies $\cap\cup X_i=\cup\cap X_i=X$ for a Borel class $X\in\mathscr{B}$ is said to be a \textit{Borel generating sequence} of $X$.
Borel generating sequences have tractable properties with regard to class operations.
The next theorem shows that the union of a couple of generating sequences of Borel classes does generate their union.

\begin{thm}\label{gen1}
  For each pair of Borel generating sequences $(X_i)_{i\in\FN}$ and $(Y_i)_{i\in\FN}$ of Borel classes $X,Y\in\mathscr{B}$, the sequence $(X_i\cup Y_i)_{i\in\FN}$ also generates $X\cup Y$.
\end{thm}

\begin{proof}
  Suppose there exists $a\in X$ which satisfies $a\notin \cup\cap(X_i\cup Y_i)$.
  Then, there always exists $j>i$ which satisfies $a\notin X_j\cup Y_j$ for all $i\in\FN$, implying that $a\notin X_j$.
  It contradicts with that $a\in \cup\cap X_i$ since $\cup\cap X_i=X$.
  Thus, $(\cup\cap X_i)\cup (\cup\cap Y_i)=X\cup Y\subseteq \cup\cap (X_i\cup Y_i)$ holds.
  
  Secondly, suppose that there exists $a\in \cap\cup (X_i\cup Y_i)$ which satisfies $a\notin X\cup Y$.
  Then, there exists $i\in\FN$ which satisfies $a\notin X_j$ and $a\notin Y_j$ for all $j>i$.
  It contradicts with that $a\in \cap\cup (X_i\cup Y_i)$.
  Thus, $\cap\cup (X_i\cup Y_i)\subseteq X\cup Y$ and, finally, the equations $\cap\cup(X_i\cup Y_i)=\cup\cap(X_i\cup Y_i)=X\cup Y$ hold.
\end{proof}

There also exist Borel generating sequences which preserve $\sigma$-additivity.
To prove the theorem, let us first state the next lemma.

\begin{lemma}\label{fingap}
  For any generating sequence $(X_i)_{i\in\FN}$ of a Borel class $X\in\mathscr{B}$, the sequence $(Y_i)_{i\in\FN}$ consisting of the same elements with $(X_i)_{i\in\FN}$ except only finite of them differ, that is, $\{i\in\FN\ ;\  Y_i\ne X_i\}$ is finite set, also generates $X$.
\end{lemma}

\begin{proof}
  By definition of $(Y_i)_{i\in\FN}$, there exists $j\in\FN$ which satisfies the equation $\cap_{i> j} X_i=\cap_{i> j} Y_i$.
  Since $\cap_{i>k} X_i \subseteq \cap_{i>j} X_i$ for all $k\in j$, it is also satisfied that $\cup\cap X_i =\cup\cap Y_i$ and, consequently,  $\cup\cap Y_i=X$.
  Similarly, the equations $\cap\cup X_i =\cap\cup Y_i$ and $\cap\cup Y_i=X$ hold. 
\end{proof}

The lemma assures generally that no matter how many elements of a generating sequence change, which Borel class it generates remains the same unless it is finite.
Consequently, the next theorem results.

\begin{thm}\label{countapp}
  Let $B$ be a Borel class and $\{B(j)\ ;\ j\in\FN\}$ be a family of Borel classes which satisfies $\cup B(j)=B$.
  Then there exists a family of generating sequences $\{(X(j)_i)_{i\in\FN}\ ;\ j\in\FN\}$ whose element generates $B(j)$ and union $\left(\cup X(j)_i\right)_{i\in\FN}$ generates $B$.  
\end{thm}

\begin{proof}
  Let $\{(Y(j)_i)_{i\in\FN}\ ;\ j\in\FN\}$ be an arbitrarily given family of generating sequences, each of which generates $B(j)$.
  Let $X(j)_i$ be set-theoretically definable class which satisfies
  \[
  X(j)_i \ =\
  \begin{cases}
    \emptyset & \text{ if }\ i<j\\
    Y(j)_i & \text{ otherwise.}
  \end{cases}
  \]
  Then, every class $X(j)_i$ is set-theoretically definable and generates $B(j)$ by Lemma \ref{fingap}. 

  Let us next confirm that $\left(\cup_{j\in\FN} X(j)_i\right)_{i\in\FN}$ is a generating sequence of $B$.
  Suppose there exists $b\in B(j)$ for some $j\in\FN$ which satisfies $b\notin \cup\cap(\cup_{j\in\FN} X(j)_i)$.
  Then, there always exists $k>i$ which satisfies $b\notin X(j)_k$ for all $i\in\FN$.
  It contradicts with that $b\in\cup\cap X(j)_i$ since $\cup\cap X(j)_i=B(j)$.
  Thus, $B \subseteq\cup\cap\left(\cup_{j\in\FN} X(j)_i\right)$ holds.

  Secondly, suppose that there exists $b\in\cap\cup(\cup_{j\in\FN} X(j)_i)$ which satisfies $b\notin B$.
  Then, there exists $i\in\FN$ which satisfies $b\notin X(j)_k$ for all $k>i$ for all $j\in\FN$.
  It implies that $b\notin \cap\cup\left(\cup_{j\in\FN} X(j)_i\right)$.
  Thus, $\cap\cup(\cup_{j\in\FN} X(j)_i)\subseteq B$ and, finally, the equations $\cap\cup(\cup_{j\in\FN} X(j)_i)=\cup\cap(\cup_{j\in\FN}(X(j)_i)=B$ hold.
\end{proof}

\section{Cuts and Their Arithmetic}\label{cuts}
To estimate size of classes and the ratio between them, let us introduce \textit{cuts} of rational numbers\footnote{Cuts are originally defined as initial segments of natural numbers by Kalina and Zlato\v{s}\cite{Kalina-Zlatos1988}.
To mention these cuts, let us call them \textit{cuts consisting of naturals}, separately.
The basic properties of natural cuts concerned here are thoroughly examined by Kalina and Zlato\v{s}\cite{Kalina-Zlatos1988, Kalina-Zlatos1989-a, Kalina-Zlatos1989-b} and Sochor\cite{sochor1988a,sochor1988b}.
}.
A class $A$ is a \textit{cut} of rationals if
\[
\left(\forall a,b\in Q^+\right)
\left(\left(a\leq b \wedge b\in A\right)\ \Rightarrow\ \left(a\in A\right)\right)
\wedge
\left(\forall a\in A\right)\left(\exists b\in A\right)\left(a< b\right)
\]
where $Q^+$ represents the class of non-negative rationals.
The order $\leq$ on the family of all cuts is given by inclusion.
Given any non-negative rational $q\in Q^+$, $\isrc{q}$ denotes a cut of $q$ given as
\[
\isrc{q}
\ =\ \left\{ x\in Q^+;\  x< q\right\}.
\]
When $q\in N$, it is said to be a \textit{cut of natural number}.

Two types of cuts, one of which is a countable union and the other is intersection, are of very importance.
Both of them are constructed with a sequence $(a_i)_{i\in\FN}$ of finite rationals which satisfies, for each $n\in \FN$, $a_j>n$ for some $j\in\FN$, respectively, as
\[
\isrc{\frac{\FN}{1}}\ =\ \cup_{i\in\FN}\isrc{a_i},\qquad\text{ and }\qquad
\isrc{\frac{1}{\FN}}\ =\ \cap_{i\in\FN}\isrc{\frac{1}{a_i}}.
\]
The class of all Borel classes of cuts, which is a $\sigma$-ring, is denoted as $\mathscr{B}_Q$.

Binary operations between any two arbitrarily given cuts $A$ and $B$ are defined as follow, where $B'$ denotes $B\cup\{\min{(Q^+\setminus B)}\}$:
\footnote{For more detailed treatment of cuts, Kalina and Zlato\v{s}\cite{Kalina-Zlatos1988,Kalina-Zlatos1989-a} and Sochor\cite{sochor1988a,sochor1988b} are helpful.}
\begin{eqnarray*}
A\upls B & = & \left\{
 x\in Q^+\,;\, \left(\forall a\notin A\right)
 \left(\forall b\notin B\right)
 \left(x< a+b\right)
\right\};\  \text{external sum}\\
A\dpls B & = & \left\{
 x\in Q^+\,;\, \left(\exists a\in A\right)
 \left(\exists b\in B\right)
 \left(x< a+b\right)
\right\};\ \text{internal sum}\\[.18cm]
A\umns B & = & \left\{
 x\in Q^+\,;\, \left(\exists b\notin B\right)
 \left(x+b\in A\right)
\right\};\  \text{external difference}\\
A\dmns B & = & \left\{
 x\in Q^+\,;\, \left(\forall b\in B'\right)
 \left(x+b\in A\right)
\right\};\  \text{internal difference}\\[.18cm]
A{\utms} B & = & \left\{
 x\in Q^+\,;\, \left(\forall a\not\in A\right)
 \left(\forall b\not\in B\right)
 \left(x< a\cdot b\right)
\right\};\  \text{external product}\\
A{\dtms} B & = & \left\{
 x\in Q^+\,;\, \left(\exists a\in A\right)
 \left(\exists b\in B\right)
 \left(x< a\cdot b\right)
\right\};\  \text{internal product}\\[.18cm]
A{\udiv} B & = & \left\{
 x\in Q^+\,;\, \left(\exists b\not\in B\right)
 \left(x\cdot b\in A\right)
\right\};\  \text{external quotient}\\
A{\ddiv} B & = & \left\{
 x\in Q^+\,;\, \left(\forall b\in B'\right)
 \left(x\cdot b\in A\right)
\right\};\  \text{internal quotient}.
\end{eqnarray*}
Both types of sum and product satisfy the commutative and associative law.
When both types of calculations 
coincide, they are simply denoted as $A+B$, $A-B$, $A\cdot B$, and $A/B$ or $\frac{A}{B}$, respectively.

Arithmetic of cuts of rational numbers coincide with that of rationals as we shall see below, however, arithmetic of cuts of non-rational number involves a somewhat strange feature, that is, \textit{additivity}.

A cut $A$ is said to be \textit{additive} if it satisfies $A+A=A$.
Otherwise, it is \textit{nonadditive}.
Two of those additives that are of great significance are $\isrc{\frac{\FN}{1}}$ and $\isrc{\frac{1}{\FN}}$ as it is mentioned above.

It is easily confirmed that arithmetic between cut of  $q\in Q$ and a cut $A$ is given as
\begin{eqnarray*}
  \isrc{q} \upls A \ =\  \isrc{q} \dpls A &  = &  \isrc{q} +A\\[.18cm]
  \isrc{q} \utms A \ =\  \isrc{q} \dtms A &  = &  \isrc{q} \cdot A\\[.18cm]
  \isrc{q} \umns A \ =\  \isrc{q} \dmns A &  = &  \isrc{q} -A, \\ 
  A \umns \isrc{q} \ =\  A \dmns \isrc{q} &  = &  A-\isrc{q}\\[.18cm]
  \isrc{q} \udiv A \ =\  \isrc{q} \ddiv A & = &  \isrc{q} / A, \\
  A \udiv \isrc{q} \ =\  A \ddiv \isrc{q} & = &  A/\isrc{q}.
\end{eqnarray*}
As long as it concerned with at least one cut of specific rational number, both types of operations coincide.
However, once both cuts are $\sigma$- or $\pi$-classes, they may differ.
To see when that happens, let us first check consistency between arithmetic of rationals and of cuts of rationals.

As it is easily confirmed, the sum of two cuts of rationals, say $\isrc{q}+\isrc{p}$, coincides with the cut of the sum of these rationals, say $\isrc{q+p}$.
This result is also valid with other three types of arithmetic.
To avoid repetition, let us demonstrate the result only for $+$ and $-$ omitting for $\cdot$ and $/$.
However, these results are valid if $+$, $-$, the identity element $0$ and the operation $p(x+y)$ are replaced with $\cdot$, $/$, $1$ and $(x\cdot y)^p$, respectively.

\begin{lemma}\label{sep-add}
  For each rational $q\in Q$ and a sequence of rationals $(a_i)_{i\in\FN}$, the following equation holds
  \[
  \isrc{q}+ A\ =\ 
  \begin{cases}
  \cup_{i\in\FN}\isrc{q+ a_i} & \text{if }A=\cup \isrc{a_i}\\
  \cap_{i\in\FN}\isrc{q+ a_i} & \text{if }A=\cap \isrc{a_i}.
  \end{cases}
  \]
\end{lemma}
\begin{proof}
  Suppose $x\in \isrc{q}+ A$.
  Then, there exist $p\in \isrc{q}$ and $a\in A$ which satisfy $x<p+ a$.
  If $A=\cup \isrc{a_i}$, there exists $j\in\FN$ which satisfies $a<\max_{i\in j}a_i$.
  It implies that $x<p+ a<q+ \max_{i\in j}a_i$, so that $x\in\cup\isrc{q+ a_i}$ holds.
  If, on the other hand, $A=\cap \isrc{a_i}$, then $a<a_i$ holds for all $i\in\FN$.
  It implies that $x<p+ a<q+ a_i$ for all $i\in\FN$, so that $x\in\cap\isrc{q+ a_i}$ holds.

  Conversely, suppose $x\in\cup\isrc{q+ a_i}$.
  Then, there exists $j\in\FN$ which satisfies $x< q+ a$, in which $a=\max_{i\in j}a_i$.
  Choose $c\in\left(0,\frac{q+ a- x}{2}\right)$.
  Then, $q- c\in \isrc{q}$, $a- c\in A$ and $x<(q- c)+(a- c)$.
  It implies that $x\in \isrc{q}+ A$.

  Suppose, on the other hand, $x\in\cap\isrc{q+ a_i}$.
  Since $p\geq q$ for all $p \not\in\isrc{q}$ and there exists $j\in\FN$ for each $a\not\in\cap\isrc{a_i}$ which satisfies $a_j<a$, the inequation $x<q+ a_j<p+ a$ holds for all $p\not\in \isrc{q}$ and $a\not\in\cap\isrc{a_i}$.
  It implies that $x\in \isrc{q}+ A$.
\end{proof}

\begin{lemma}\label{sep-sub}
  For each rational $q\in Q$ and a sequence of rationals $(a_i)_{i\in\FN}$, the following equation holds
  \[
  \isrc{q}- A\ =\ 
  \begin{cases}
  \cup_{i\in\FN}\isrc{q- a_i} & \text{if }A=\cap \isrc{a_i}\\
  \cap_{i\in\FN}\isrc{q- a_i} & \text{if }A=\cup \isrc{a_i}.
  \end{cases}
  \]
\end{lemma}
\begin{proof}
  Suppose $x\in \isrc{q}- A$.
  Then, there exist $a\not\in A$ which satisfies $x+ a<q$.
  If $A=\cap \isrc{a_i}$, there exists $j\in\FN$ which satisfies $\min_{i\in j}\{a_i\}<a$.
  It implies that $x+ \min_{i\in j}\{a_i\}<x+ a<q$, so that $x\in\cup\isrc{q- a_i}$ holds.
  If, on the other hand, $A=\cup \isrc{a_i}$, then $a_i<a$ holds for all $i\in\FN$.
  It implies that $x+ a_i<x+ a<q$ for all $i\in\FN$, so that $x\in\cap\isrc{q- a_i}$ holds.

  Conversely, suppose $x\in\cup\isrc{q- a_i}$.
  Then, there exists $j\in\FN$ which satisfies $x+ a_j<q$, so that $x\in \isrc{q}- A$ holds
  since $a_j\not\in\cap\isrc{a_i}=A$,

  Suppose, on the other hand, $x\in\cap\isrc{q- a_i}$.
  Then, for each $a\in A$ there exists $j\in\FN$ which satisfies $a<a_j$ since $A=\cup\isrc{a_i}$.
  It implies that $x+ a<x+ a_j<q$, so that $x\in \isrc{q}- A$.
\end{proof}

\begin{lemma}
  For each rational $q\in Q$ and a sequence of rationals $(a_i)_{i\in\FN}$, the following equation holds
  \[
  A- \isrc{q}\ =\ 
  \begin{cases}
  \cup_{i\in\FN}\isrc{a_i- q} & \text{if }A=\cup \isrc{a_i}\\
  \cap_{i\in\FN}\isrc{a_i- q} & \text{if }A=\cap \isrc{a_i}.
  \end{cases}
  \]
\end{lemma}
\begin{proof}
  Suppose $x\in A- \isrc{q}$ and $A=\cup\isrc{a_i}$.
  Then, there exist $p\geq q$ and $j\in\FN$ which satisfy $x+ p< a_j$.
  Since $x+ q\leq x+ p<a_j$, $x\in\cup\isrc{a_i- q}$ holds.
  If, on the other hand, $A=\cap \isrc{a_i}$.
  Then, there exist $p\geq q$ which satisfies $x+ p< a_i$ for all $i\in\FN$.
  Since $x+ q\leq x+ p<a_i$, so that $x\in\cap\isrc{a_i- q}$ holds.

  Conversely, suppose $x\in\cup\isrc{a_i- q}$.
  Then, there exists $j\in\FN$ which satisfies $x+ q<a_j$, so that $x\in A-\isrc{q}$ holds
  since $q\not\in\isrc{q}$.

  Suppose, on the other hand, $x\in\cap\isrc{a_i- q}$.
  Then, the inequation $x+ q<a_i$ holds for all $i\in\FN$ since $A=\cap\isrc{a_i}$.
  It implies that $x+ q<a_i$ for all $i\in\FN$, so that $x\in A-\isrc{q}$ since $q\notin \isrc{q}$.
\end{proof}

By these lemmata, it is confirmed that cuts of rational numbers can be calculated exactly the same manner as rationals.
Accordingly, cuts of rationals will be described simply as rationals without mentioning, hereafter, to ease the notation unless there are any confusions.

However, these equalities are no longer valid once both cuts are $\sigma$- or $\pi$-cuts, typically when they are additive.
By Lemma \ref{sep-add} and \ref{sep-sub}, every $\sigma$- or $\pi$-cuts can be separated into two cuts: a cut of rational number and an additive one, say $\alpha+\frac{\FN}{1}$ or $\alpha-\frac{\FN}{1}$.
And additive cuts, say $\frac{\FN}{1}$ in this case, is the source of the problem.
For example, the following internal/external sums are unequal for any rational cut $\alpha>\frac{\FN}{1}$:
\[\textstyle
\alpha-\frac{\FN}{1}\dpls\frac{\FN}{1}\ =\
\alpha-\frac{\FN}{1}\ \ne\ \alpha+\frac{\FN}{1}
\ =\ \alpha-\frac{\FN}{1}\upls\frac{\FN}{1}.
\]

This inequality exposes the unique characteristics of additive cuts to smooth off minute differences unless they exceed the size of the cut itself.
In fact, even if the cut $\frac{\FN}{1}$ is added internally to $\alpha-\frac{\FN}{1}$, the result remain stable at $\alpha-\frac{\FN}{1}$ since every $b\in\frac{\FN}{1}$ and $c\in \alpha-\frac{\FN}{1}$ satisfy $b+c\in\alpha-\frac{\FN}{1}$.
The least cut which is smaller than any other cuts obtained by adding infinite cuts to $\alpha-\frac{\FN}{1}$ is the external sum $\alpha-\frac{\FN}{1}\upls\frac{\FN}{1}=\alpha+\frac{\FN}{1}$.

This property of additive cuts, which absorb small cuts, can be utilized to capture equivalence of cuts.
To formulate the equivalence, let us first define the interior and closure\footnote{
  Equivalently, they can be defined as $\intrr{Q}{A}=Q^+\setminus \cl{Q}{A^c}$ and $\cl{Q}{A}=\fig{Q}{A}$ of a quasi-continuum $\mathscr{Q}=\langle Q^+,\doteq_Q\rangle$, in which $\doteq_Q$ is define as $a\doteq_Q b$ iff $a=b=0$ or $\left((a,b\ne 0)\wedge(\frac{a}{b}\doteq 1)\right)$.
  Notice that $\mathscr{Q}=\langle Q,\doteq_Q\rangle$ remains quasi-continuum since $\doteq_Q$ is not compact.
  See Sakahara and Sato \cite{topology} for more details.} of a cut $A$ as
\begin{eqnarray*}
\intrr{}{A} & = &\textstyle A\dmns A\ddiv\frac{\FN}{1};\\
\cl{}{A} & = &\textstyle A\dpls A\ddiv{\frac{\FN}{1}}.
\end{eqnarray*}
Then, a pair of cuts are said to be equivalent when one cut is in between the other's interior and closure, denoted by equivalence $\approx$ as
\[
A\approx B\ \equiv\ \intrr{}{A}\leq B\leq \cl{}{A}.
\]
%
In other words, two cuts are equivalent if differences between the two is smaller than the biggest additive cuts which cannot absorb them.
In fact, the difference between internal and external sum of $\alpha-\frac{\FN}{1}$ and $\frac{\FN}{1}$ is negligible  since $\alpha/\frac{\FN}{1}> \frac{\FN}{1}$ for all $\alpha\in N\setminus\FN$ and, thus, $\alpha-\frac{\FN}{1}\approx\alpha+\frac{\FN}{1}$. 
Needless to say that $\alpha/\frac{\FN}{1}$ is the maximal additive cut which does not absorb  $\alpha$.

The properties which can be drawn from these characteristics of additive cuts are summarized in the following lemmata, which are extensions of the statement (5) a) of Sochor \cite{sochor1988a}.
\begin{lemma}\label{s5a-1}
  For any additive cut $A$, a rational cut $\alpha>A$ and a cut $B<A$, the following equations hold:
  \[
  \alpha+A+B\ =\ \alpha+A\ =\ \alpha+A-B.
  \]
\end{lemma}
\begin{proof}
  Let us first confirm that $\alpha+A\dpls B= \alpha+A\upls B=\alpha+A$ hold.
  Since $\alpha+A\subseteq \alpha+A\dpls B\subseteq \alpha+A\upls B$, it is sufficient to show that $\alpha+A\upls B\subseteq \alpha+A$.
  Let $x\in \alpha+A\upls B$.
  By definition of external sum, $x<\alpha+a+b$ holds for all $a\notin A$ and $b\notin B$.
  To show that $x\in \alpha+A$, it is sufficient to confirm that $x<\alpha+a$ for all $a\notin A$.
  Since $A$ is additive, $a/2\notin A$ holds.
  It is also true that there exists $b\notin B$ which satisfies $b<a/2$, since $B<A$.
  It implies that $x<\alpha+a/2+b<\alpha+a$ for all $a\notin A$.
  Consequently, $\alpha+A\upls B\subseteq \alpha+A$ holds.

  Let us next confirm that $\alpha+A=\alpha+A\umns B=\alpha+A\dmns B$ hold.
  Since $\alpha+A\umns B\subseteq \alpha+A\dmns B\subseteq \alpha+A$, it is sufficient to show that $\alpha+A\subseteq \alpha+A\umns B$.
  Let $x\in \alpha+A$.
  Then, $x<\alpha+a$ is satisfied for all $a\notin A$ by definition of external sum.
  Since $A$ is additive, it is satisfied that $a/2\notin A$ for all $a\notin A$.
  Since $B<A$, there exists $a\notin A$ which satisfies $a/2>b$ for each $b\notin B$.
  It implies that for each $b\notin B$ there exists $a\notin A$ which satisfies $x<\alpha+a/2<\alpha+ a-b$.
  Consequently, $x\in \alpha+A\umns B$ and $\alpha+A\subseteq \alpha+A\umns B$ holds.
\end{proof}
The next lemma also follows by the same argument. 
\begin{lemma}\label{s5a-2}
  For any additive cut $A$, and a cut of rational $\alpha> A$ and a cut $B<A$, the following equations hold:
  \[
  \alpha-A+B\ =\ \alpha-A\ =\ \alpha-A-B.
  \]
\end{lemma}
\begin{proof}
  Let us first confirm that $\alpha-A\dpls B= \alpha-A\upls B=\alpha-A$ hold.
  Since $\alpha-A\subseteq \alpha-A\dpls B\subseteq \alpha-A\upls B$, it is sufficient to show that $\alpha-A\upls B\subseteq \alpha-A$.
  Let $x\in \alpha-A\upls B$.
  Then $x<\alpha-a+b$ for all $a\in A$ and $b\notin B$ by definition of external sum.
  Since $A$ is additive, it is also met that $2\cdot a\in A$ for all $a\in A$. 
  Since $B<A$, there exists $b\notin B$ which satisfies $b<a$ for each $a\in A$.
  It implies that $x<\alpha-2\cdot a+b<\alpha- a$. 
  Consequently, $\alpha-A\dpls B\subseteq \alpha-A$ holds.

  Let us next turn to confirm that $\alpha-A=\alpha-A\umns B=\alpha-A\dmns B$ hold.
  Since $\alpha-A\umns B\subseteq \alpha-A\dmns B\subseteq \alpha-A$, it is sufficient to show that $\alpha-A\subseteq \alpha-A\umns B$.
  Let $x\in \alpha-A$.
  Since $A$ is additive, it is satisfied that $2\cdot a\notin A$ for all $a\notin A$, and thus, $x+2\cdot a\in \alpha$ for some $a\notin A$ by definition of external difference $\alpha\umns A$.
  Since $B<A$, there exists $b\notin B$ which satisfies $b<a$, and thus, $x+a+b\in \alpha$ since $x+a+b<x+2\cdot a$ holds.
  Consequently, $\alpha-A\subseteq \alpha-A\umns B$ follows.
\end{proof}

The cuts which play a central role in the present paper are of the form $\alpha\pm A$ in which $\alpha$ is a cut of rationals and $A$ is additive.
So, let us confirm that arithmetic between these types of cuts.

\begin{lemma}\label{mm}
  For any given $\alpha,\beta\in Q^+$ and additive cuts $A>B$, the following equations are met.
  \begin{eqnarray*}
    (\alpha-A)+(\beta-B) & = &
    \begin{cases}
      \alpha-A & \text{ if }\beta<A\\
      \alpha+\beta-A & \text{ otherwise}
    \end{cases}
  \end{eqnarray*}
\end{lemma}
\begin{proof}
  Since $(\alpha-A)\dpls(\beta-B)\subseteq (\alpha-A)\upls(\beta-B)\subseteq \alpha+\beta-A$, it is sufficient to show that $(\alpha-A)\dpls(\beta-B)=\alpha+\beta-A$.
  It is proven by the statement (c) of (19) of Sochor \cite{sochor1988a}.
  When $\beta<A$, the equation $\alpha+\beta-A=\alpha-A$ follows by Lemma \ref{s5a-2}.
\end{proof}

\begin{lemma}\label{pmmp}
  For any given $\alpha,\beta\in Q^+$ and additive cuts $A>B$, the following equation holds:
  \begin{eqnarray*}
    (\alpha\pm A)+(\beta\mp B) & = &
    \begin{cases}
      \alpha\pm A & \text{ if }\beta<A\\
      \alpha+\beta\pm A & \text{ otherwise.}
    \end{cases}
  \end{eqnarray*}
\end{lemma}
\begin{proof}
  Since $(\alpha\pm A)\dpls(\beta\mp B)\subseteq (\alpha\pm A)\upls(\beta\mp B)\subseteq \alpha+\beta\pm A$, it is sufficient to show that $(\alpha\pm A)\dpls(\beta\mp B)=\alpha+\beta\pm A$.
  It is proven by the statement (b) of (19) of Sochor \cite{sochor1988a}.
  When $\beta<A$, the equation $\alpha+\beta\pm A=\alpha\pm A$ follows by Lemma \ref{s5a-1} and \ref{s5a-2}.
\end{proof}

Lastly, countable sum of cuts are defined as follows.
Let $(A_i)_{i\in\FN}$ be a sequence of cuts.
Then the two types of countable sum, say, lower bound and upper bound, are defined as
\begin{eqnarray*}
\udot{\sum}_{i\in\FN} A_i
& = &
\left\{x\,;\, 
  \left(\exists j\right)
  \left(x\in A_0+A_1+\cdots+A_j\right)\right\},\\
\dot{\sum}_{i\in\FN} A_i 
& = &
\left\{x\,;\, 
  \left(\forall f\right)
  \left(
  \left(
  \begin{matrix}
    \left(\FN\subseteq\dom(f)\right)\\
    \wedge \left(\rng(f)\subseteq Q^+\right)\\
    \wedge \left(\forall i\right)
  \left(f(i)\notin A_i\right)
  \end{matrix}
  \right)
  \ \Rightarrow \ x<\sum f(i)
  \right)
\right\}.
\end{eqnarray*}

\section{Measuring Size of Borel Classes}

Let us next confirm a way to ascertain the size of classes with cuts of naturals.
To arrange cuts in order, let us introduce subvalent relations.
$n$ is said to be \textit{strictly subvalent to} $u$, in symbol $n \widehat{\prec} u$, if there exists an injective set function $f$ whose domain is $n$ and range is a proper subset of $u$.

For each class $X\subseteq V$, its \textit{lower cut} $\underline{X}$ and \textit{upper cut} $\overline{X}$\footnote{These cuts are originally defined also as initial segments of natural numbers by Kalina and Zlato\v{s}\cite{Kalina-Zlatos1988} too.
However, this difference doesn't matter since natural numbers contained in rational numbers can still count the numbers.} are defined as
\begin{eqnarray*}
\underline{X} & \equiv & 
\left\{
 q\in Q^+\,;\, \left(\exists n\in N)(\exists u\right)
 \left(u\subseteq X\wedge n\widehat{\prec} u\right) \wedge \left( q<n\right)
\right\},\\
\overline{X} & \equiv & 
\left\{
 q\in Q^+\,;\, \left(\exists n\in N)(\forall u\right)
 \left(u\supseteq X \Rightarrow n\widehat{\prec} u\right) \wedge \left( q<n\right)
\right\}.
\end{eqnarray*}
If $\underline{X}=\overline{X}$, then the common value is denoted by $|X|$ and called the \textit{cut} of $X$.

The lower and upper cuts of a given class differ when, for example, $X$ consists of $X_0=\alpha\setminus A$ and $X_1=\{0\}\times A$ where $A<\alpha$ is an additive cut.
Then, its lower and upper cuts are given as $\underline{X}=\alpha-A$ and $\overline{X}=\alpha+A$, respectively.

Fundamental properties between class operations and arithmetic of cuts are outlined in the next proposition.

\begin{prop}[Proposition 3.1.6 of Kalina and Zlato\v{s} \cite{Kalina-Zlatos1988}]\label{KZ315}
  For all classes $X,Y$ the following assertions hold:
  \begin{enumerate}[(a)]
  \item $X\cap Y=\emptyset\ \Rightarrow\ \underline{X}\dpls\underline{Y}\leq \underline{X\cup Y}$,
  \item $\overline{X\cup Y}\leq \overline{X}\upls\overline{Y}$,
    \item if there is a set-theoretically definable class $S$ such that $X\subseteq S$, $Y\cap S=\emptyset$, in short $X$ and $Y$ are separable $\Sep(X,Y)$, then $\underline{X\cup Y}=\underline{X}\dpls\underline{Y}$ and $\overline{X\cup Y}=\overline{X}\upls\overline{Y}$. 
  \end{enumerate}
  \end{prop}

The strict case of the first two conditions happens when $X$ and $Y$ are not separable.
Let $X=(\alpha+\beta)\setminus A$, where $\beta<A<\alpha$ and $A$ is additive, and $Y=A$, for example, then $X\cup Y=\alpha+\beta$, so that $|X\cup Y|=\alpha+\beta$.
Nevertheless, $|X|\dpls|Y| =\alpha-A$ and $|X|\upls|Y|=\alpha+A$, so that $\underline{X}\dpls\underline{Y} <\underline{X\cup Y}$ and $\overline{X}\upls\overline{Y} >\overline{X\cup Y}$ hold.

However, these differences are negligible as far as the Borel classes are dealt since the following theorem holds.

\begin{thm}[Theorem 3.2. of Kalina and Zlato\v{s} \cite{Kalina-Zlatos1989-b}]\label{KZ32}
  Let $X\in\mathscr{B}$ and $\alpha\in N$.
  If $\underline{X}\leq \alpha\leq\overline{X}$, then
  \[
  \intrr{}{\alpha}\leq \underline{X}\leq \alpha\leq\overline{X}\leq\cl{}{\alpha}.
  \]
\end{thm}

Consequently from the theorem, The following basic properties\footnote{For more detailed properties of cuts of Borel semisets, see  \cite{Kalina-Zlatos1989-b}, \cite{Kalina1989-a} and \cite{Kalina1989-b}} of Borel classes are drawn.

\begin{cor}[Corollary 3.3. of Kalina and Zlato\v{s} \cite{Kalina-Zlatos1989-b}]\label{KZ33}
  Let $X$ be a Borel class.
  \begin{enumerate}[(a)]
  \item If either $\underline{X}$ or $\overline{X}$ is an additive cut, then $\underline{X}=\overline{X}$.
  \item If $\underline{X}$ and $\overline{X}$ are nonadditive, then there exists  $a\in N$ which satisfies $\intrr{}{a}\leq\underline{X}\leq\overline{X}\leq\cl{}{a}$.
    If additionally $\underline{X}\ne\overline{X}$, then each $a\in\overline{X}\setminus\underline{X}\cap N$ satisfies the inequality.
  \item $\underline{X}\approx\overline{X}$.
  \end{enumerate}
\end{cor}

Furthermore, the next Theorem is especially important.
\begin{thm}[Theorem 3.4 of Kalina and Zlato\v{s} \cite{Kalina-Zlatos1989-b}]\label{KZ34}
    Let $X$ be an arbitrarily given class.
    Then, $X\in\mathscr{B}$ iff it is satisfied that $|X|$ exists and $|X|$ is a $\sigma$- or $\pi$-cut, or there exists an additive cut $A<\underline{X}$ which satisfies $\underline{X}=n-A$ and $\overline{X}=n+A$ for each $n\in\overline{X}\setminus \underline{X}\cap N$.
\end{thm}

This theorem is very powerful to state that there exists Borel class $Y$ for each Borel class $X$ such that
\[
Y\ =\ \underline{X}\cup\left(\left(\overline{X}\setminus \underline{X}\right)\times \{0\}\right).
\]
Provided that $\underline{X}=n-A$, $\overline{X}=n+A$ and $A$ is additive, the next equation follows:
\[
Y\ =\ (n-A)\cup (A\times\{0\}).
\]
Since additive cut is a $\sigma$- or $\pi$-class, every Borel class has at least one $\delta_2$-class which has the same pair of cuts as $X$.
The next theorem states this fact.
\begin{thm}[Theorem 3.5 of Kalina and Zlato\v{s} \cite{Kalina-Zlatos1989-b}]\label{KZ35}
For every Borel class $X\in\mathscr{B}$, there is a class $Y$ which is a union of a $\pi$-class and a $\sigma$-class such that $\underline{X}=\underline{Y}$ and $\overline{X}=\overline{Y}$.
\end{thm}

As we saw in Section \ref{BC}, every Borel class can be built as $\delta_2$-class by Theorem \ref{pisigma}.
Theorem \ref{KZ35} can be deduced directly from the result since cuts of $\delta_2$-classes are given as union of $\pi$- and $\sigma$-classes.

Let us next consider a kind of generating sequences of cuts.
A sequence $(s_i)_{i\in \FN}$ of cuts of natural numbers is said to be an \textit{approximating sequence}\footnote{\label{approxseq}Notice that while Kalina \cite{Kalina1989-a} originally defined it as an uncountable sequence, the present paper modified it to be a countable one.
  It may seem insignificant, though, this tiny alteration brings about the key structure allowing us to establish $\sigma$-superadditive measures.
} of a pair of cuts $\langle A,B\rangle$ if
\[
\bigcup\bigcap\ s_i\ =\ A\quad \text{and}\quad
\bigcap\bigcup\ s_i\ =\ B.
\]
In particular, $(s_i)_{i\in \FN}$ is said to be an \textit{approximating sequence} of a class $X$ if it is an approximating sequence of the pair $\langle\underline{X},\overline{X}\rangle$.

Notice first that, since approximating sequences are composed of cuts of natural numbers, it is not possible to approximate set-theoretical proper classes which have infinite cuts, say $\cup_{i\in N}\isrc{i}$, so that only the semisets have their approximating sequences.
Furthermore, as it is shown in the next lemma, every Borel semiset has its approximating sequence.

\begin{lemma}[Lemma 2.2 of Kalina \cite{Kalina1989-a}]
  Let $B$ be a Borel semiset.
  Then, there exists an approximating sequence of $B$.
\end{lemma}

It may seem that a sequence consisting of cuts $|X_i|$ of elements of a generating sequence $(X_i)_{i\in\FN}$ of a given Borel class $X$ approximates $X$ as well.
However, it is not the case.
An example is given as follow: 
Let $u$ be a set and  $(a_i)_{i\in\FN}$ be a sequence whose elements are given as
\[
a_i\ =\ u\cup \{x_i\}\quad\text{ where } x_i\ne x_j\text{ for all }j\ne i.
\]
Then, $(a_i)_{i\in\FN}$ is a Borel generating sequence of a set $u$.
But $(|a_i|)_{i\in\FN}$ is not a Borel approximating sequence of $|u|$ since $\cup\cap|a_i|=\cap\cup|a_i|=|u|+1$.
For $(|X_i|)_{i\in\FN}$ to be an approximating sequence of a Borel class $X$, $(X_i)_{i\in\FN}$ must be divided into monotone generating sequences  $(X(k)_i)_{i\in\FN}$ of $X(k)$ in which $\cup_{k\in \{0,1\}} X(k)=X$ where $X(0)$ is a $\sigma$-class and $X(1)$ is a $\pi$-class. 
\begin{thm}\label{monotone-BGF}
  Let $(X(0)_i)_{i\in\FN}$ and $(X(1)_i)_{i\in\FN}$ be monotone generating sequences of a $\sigma$-class $X(0)$ and a $\pi$-class $X(1)$, respectively, in which $X(0)$ and $X(1)$ are separated.
  Then, $(|X(0)_i\cup X(1)_i|)_{i\in\FN}$ approximates $X(0)\cup X(1)$.
\end{thm}

\begin{proof}
  Since $X(0)$ and $X(1)$ are separated, there exists $j\in\FN$ which satisfies $X(0)_i\cap X(1)_i=\emptyset$ for all $i>j$.
  Since $(X(0)_i)_{i\in\FN}$ and $(X(1)_i)_{i\in\FN}$ are monotone, the equation $\cap_{i\geq k}|X(0)_i\cup X(1)_i|=|X(0)_k\cup \cap_{i\in\FN} X(1)_i|$ holds for all $k>j$ and, thus, for all $k\in\FN$ since $X(0)_k$ is increasing.
  Furthermore, the following equations hold
  \begin{eqnarray*}
    \cup_{k\in\FN}|X(0)_k\cup \cap_{i\in\FN} X(1)_i| & = & \cup_{k\in\FN}|X(0)_k|\underset{\cdot}{+} |X(1)|\\
   & = & |\cup_{k\in\FN}X(0)_k|\underset{\cdot}{+} |X(1)|\ =\ |X(0)|\underset{\cdot}{+}|X(1)|.
  \end{eqnarray*}
  Since $X(0)$ and $X(1)$ is separable, $|X(0)|\underset{\cdot}{+}|X(1)|=\underline{X}$ holds.

  There also exists $j\in\FN$ which satisfies $\cup_{i\in\FN}( X(0)_i)\cap X(1)_j =\emptyset$ since $X(0)$ and $X(1)$ are separated.
  Since $(X(0)_i)_{i\in\FN}$ and $(X(1)_i)_{i\in\FN}$ are monotone, the equation $\cup_{i\geq k}|X(0)_i\cup X(1)_i|=|\cup_{i\in\FN} (X(0)_i)\cup X(1)_k|$ holds for all $k>j$.
  Furthermore, the following equations hold
  \begin{eqnarray*}
    \cap_{k\in\FN}|\cup_{i\in\FN}X(0)_i\cup X(1)_k| & = & |X(0)|\overset{\cdot}{+} \cap_{k\in\FN}|X(1)_k|\\
   & = & |X(0)|\overset{\cdot}{+} |\cap_{i\in\FN}X(1)_i|\ =\ |X(0)|\overset{\cdot}{+}|X(1)|.
  \end{eqnarray*}
  Since $X(0)$ and $X(1)$ is separable, $|X(0)|\overset{\cdot}{+}|X(1)|=\overline{X}$ holds.
\end{proof}

By Theorem \ref{monotone-BGF}, every Borel semiset $X$ is guaranteed to have a generating sequence $(X_i)_{i\in\FN}$ which consists of the elements whose cuts form an approximating sequence of a pair of cuts of $X$, since every Borel semiset can be separated into a $\sigma$-class $X(0)$ and a $\pi$-class $X(1)$ by Theorem \ref{KZ34}.

Notice that the system of all Borel semisets $\mathscr{B}_s$ is not a class of extended universe, since it consists not only of sets but also of proper classes.
In order to treat the sytem in the same manner as sets, it is helpful to remind codable classes and postulate the axiom of extensional coding, which is equivalent to the axiom of choice. 

A family of classes $\mathscr{A}$ is said to be \textit{codable} if there exists a pair of classes $\langle X,S\rangle$ such that
\begin{equation}\label{codability}
\left(\forall Y\in\mathscr{A}\right)
\left(\exists y\in X\right)
\left(S``y=Y\right) \wedge 
\left(\forall y\in X\right)
\left(S``y\in\mathscr{A}\right)
\end{equation}
and the pair $\langle X,S\rangle$, having property (\ref{codability}) is said to be the \textit{coding pair} of $\mathscr{A}$.

\begin{remark}
Since  $\mathscr{B}_s$ is the smallest $\sigma$-ring such that $V\subset \mathscr{B}_s$, $\mathscr{B}_s$ is codable.
\end{remark}

\newtheorem*{ax03}{Axiom of extensional coding}
\begin{ax03}
Each codable family of classes $\mathscr{A}$ is \textit{extensionally codable}, i.e., there exists such a coding pair $\langle X,S\rangle$ of $\mathscr{A}$, for which
\begin{equation}\label{extensionality}
\left(\forall x,y\in X\right)
\left( S``x=S``y\ \equiv\ x=y\right)
\end{equation}
holds.
\end{ax03}

\begin{lemma}[Lemma 2.4 of Kalina \cite{Kalina1989-a}]
Let $\langle X,S\rangle$ be a coding pair of $\mathscr{B}_s$.
Then there exists a map $G$ with $\dom(G)=X$ such that for each $x\in X$, $G(x)$ is a Borel generating sequence of $S``x$.
\end{lemma}

We shall consider $\mathscr{B}_s$ as the domain of the map $F$ and the value of $F(B)=G(x)$ for each $B\in\mathscr{B}_s$ in which $S``x=B$ for a coding pair $\langle X,S\rangle$ of $\mathscr{B}_s$. 
Any map $F$ which assigns a sequence generating $A$, guaranteed by Theorem \ref{monotone-BGF}, is called the \textit{Borel generating function}  (\textit{BGF} for short), and a map $|F|$, or simply $F$, which assigns to each class $A\in\mathscr{B}_s$ a Borel approximating sequence of $A$, is called the \textit{Borel approximating function} (\textit{BAF}).

\section{Comparing Size of Borel Classes}

To assess the ratio of one cut $B$ of a given Borel class to the other $S$ in terms of cuts, one problem arises.
As already mentioned in Section \ref{cuts} that internal and external calculations differ when additive cuts are dealt, the value of cuts by division is also not uniquely determined in general, specifically in the case where an additive cut is divided by another one.
To ascertain the ratios, approximating sequences are utilized.

Let $B$ and $S$ be Borel semisets, and $(b_i)_{i\in\FN}$ and $(s_i)_{i\in\FN}$ be approximating sequences of $B$ and $S$, respectively.
Let $\varliminf_{i\in\FN} {\frac{b_i}{s_i}}$ denote a \textit{limit inferior} of a sequence of rational cuts $\bigl({\frac{b_i}{s_i}}\bigr)_{i\in\FN}$ defined as
\[
\varliminf_{i\in\FN} {\frac{b_i}{s_i}}
\ =\ \bigcup_{i\in\FN}\bigcap_{j\geq i} 
{\frac{b_j}{s_j}}.
\]
Similarly, let $\varlimsup_{i\in\FN}\ {\frac{b_i}{s_i}}$ denote a \textit{limit superior}  given as
\[
\varlimsup_{i\in\FN} {\frac{b_i}{s_i}}
\ =\ \bigcap_{i\in\FN}\bigcup_{j\geq i} 
{\frac{b_j}{s_j}}.
\] 
If $\varliminf_{i\in\FN} {\frac{b_i}{s_i}}=\varlimsup_{i\in\FN} {\frac{b_i}{s_i}}$, the common value is denoted simply as
$
\lim_{i\in\FN}\,{\frac{b_i}{s_i}}
$
and called the \textit{limit} of ratio of $B$ to $S$.

Moreover, when approximating sequences are given by a BAF $F$, that is, $F(B)=(b_i)_{i\in\FN}$ and $F(S)=(s_i)_{i\in\FN}$, a sequence $\bigl(\frac{b_i}{s_i}\bigr)_{i\in\FN}$ is simply denoted as $\frac{F(B)}{F(S)}$, or $\frac{(b_i)}{(s_i)}$, and its limit is also denoted, equivalently, as
\[
\lim\,{\frac{F(B)}{F(S)}}
\ =\ 
\lim\,{\frac{(b_i)}{(s_i)}}
\ =\ 
\lim_{i\in\FN}\,{\frac{b_i}{s_i}}.
\]
In case the limit is independent of choice of a BAF, it is denoted simply as ${\frac{|B|}{|S|}}$.

Let us expand indiscernibility equivalence $\doteq$ on Borel cuts.
Borel cuts $A$ and $B$ are  equivalent iff they satisfy the following conditions:
\[
A\doteq
B\quad\Leftrightarrow\quad
\left(
(A\dmns B)\cup(B\dmns A)\subseteq {\frac{1}{\FN}}
\right)
\vee
\left(
A\cap B\supseteq {\frac{\FN}{1}}
\right).
\]
It is evident that $\isrc{q}\doteq
\isrc{p}$ iff $q\doteq p$ for any pair of $p,q\in Q$.

The Borel semiset $B$ is said to be \textit{$\langle s,F\rangle$-observable} if its limits inferior and superior of ratio of $B$ to $S$ are equivalent, that is, $\varliminf{\frac{b_i}{s_i}}\doteq
\varlimsup{\frac{b_i}{s_i}}$ where $F$ is a BAF and $F(B)=(b_i)_{ i\in \FN}$, and $s=(s_i)_{ i\in \FN}$ is an approximating sequence of $S\in\mathscr{B}_s$.
The class of all $\langle s,F\rangle$-observable semisets is denoted by $\mathscr{O}(s,F)$.
\medskip

When $S$ is a semiset having a nonadditive cut, the ratio of any given Borel semiset $B$ to $S$ has a limit and is uniquely determined up to the equivalence $\doteq
$ regardless of choice of a BAF.
The following statement is essentially the same as Theorem 3.3 of Kalina \cite{Kalina1989-a}.

\begin{thm}\label{uniqueness}
  Let $S$ be a Borel semiset which has a nonadditive cut $|S|$ and whose approximating sequence is $(s_i)_{i\in\FN}$, then for any Borel semiset $B$ and a BAF $F$, $\mathscr{O}(s,F)=\mathscr{B}_s$.
  Moreover, if $G$ is any other BAF, the following equation holds.
\[
\varliminf {\frac{F(B)}{(s_i)}}\ \doteq
\ \varliminf{\frac{G(B)}{(s_i)}}\quad 
\]
When $B$ has an additive cut, then the following inclusions hold:
\begin{enumerate}[i.)]
\item if $|B|\subset \cup\cap(s_i)_{i\in\FN}$ then $\varlimsup{\frac{F(B)}{(s_i)}}\subseteq {\frac{1}{\FN}}$.
\item if $|B|\supset \cap\cup(s_i)_{i\in\FN}$ then $\varliminf{\frac{F(B)}{(s_i)}}\supseteq {\frac{\FN}{1}}$,
\end{enumerate}
\end{thm}

To prove Theorem \ref{uniqueness}, the next proposition is useful, which is essentially the same as Proposition 3.1 of Kalina \cite{Kalina1989-a}.
\begin{prop}\label{limit}
Let $(b_i)_{i\in\FN}$ be an arbitrarily given approximating sequence of a nonempty Borel semiset $B$.
Then the lower and upper cuts of $B$ is nonadditive iff it is satisfied that for all $\ell\in\FQ$ there exists $i\in\FN$ which satisfies $\left|\frac{b_j}{b_i}-1\right|<|\ell|$ for all $j>i$. 
\end{prop}
\begin{proof}
If $B$ has an additive cut, by (a) of Corollary \ref{KZ33}, $\underline{B}=\overline{B}$.
Since $|B|$ is additive, for each $b_i\in|B|$ there must be $j>i$ such that $\frac{b_j}{b_i}>2$.
Otherwise, for each $b_i\notin|B|$ there must be $j>i$ such that $\frac{b_j}{b_i}<\frac{1}{2}$.

If $B$ has a nonadditive cut, then by (b) of Corollary \ref{KZ33} there exists $d\in N$ such that $\intrr{}{d}\leq \underline{B}\leq\overline{B}\leq \cl{}{d}$.
Let $\ell\in\FQ$ be any given finite rational number, and $k\in\FN$ satisfies $\frac{2}{k-1}\leq|\ell|$.
Then, for this $d$ and $k$ there exists $i$ which satisfies for each $j>i$ the next inequations
\[
d-\frac{d}{k}\ <\ b_i\ <\  d+\frac{d}{k}\quad\text{ and }\quad
d-\frac{d}{k}\ <\ b_j\ <\  d+\frac{d}{k}.
\]
Then the next inequations follow:
\[
\frac{d-\frac{d}{k}}{d+\frac{d}{k}}\ =\ 
1-\frac{2}{k+1}\ <\ \frac{b_j}{b_i}\ <\ 
1+\frac{2}{k-1}\ =\ 
\frac{d+\frac{d}{k}}{d-\frac{d}{k}}.
\]
Since $\frac{2}{k-1}\leq|\ell|$, $\left|\frac{b_j}{b_i}-1\right|<|\ell|$ is
 satisfied.
\end{proof}
\begin{proof}[Proof of Theorem \ref{uniqueness}]
Let $F(B)=(b_i)_{i\in\FN}$.
If $B$ has a nonadditive cut, then by Proposition \ref{limit}, $\left|\frac{b_j}{b_i}-\frac{s_j}{s_i}\right|< \left|\ell\right|$ for all $\ell\in\FQ$ holds.
It implies that if $\frac{b_i}{s_j}\in\BQ$ for all $j>i$ for some $i\in\FN$, it is satisfied that  $\left|\frac{b_j}{s_j}-{\frac{b_i}{s_i}}\right|<\frac{b_i}{s_j}|\ell|$ for all $\ell\in\FQ$, so that $\varliminf{\frac{(b_i)}{(s_i)}}\doteq
\varlimsup{\frac{(b_i)}{(s_i)}}$ holds.
Otherwise, the equivalence is satisfied trivially since $\frac{b_i}{s_j}\in Q\setminus\BQ$.

Secondly, let us show that these rational cuts are uniquely determined.
Let $F,\,G$ be two different BAFs in which $F(B)=(o_i)_{i\in\FN}$ and $G(B)=(e_i)_{i\in\FN}$.
Define
$H(B)=(b_i)_{i\in\FN}$ in which $b_i=o_i$ if $i$ is odd and $b_i=e_i$ otherwise.
By definition of $H$ and Proposition \ref{limit}, $H$ is also a BAF and they are equivalent:
\[
\varliminf{\frac{F(B)}{(s_i)}}\ \doteq
\ 
\varliminf{\frac{H(B)}{(s_i)}}\ \doteq
\ 
\varliminf{\frac{G(B)}{(s_i)}}.
\]

If, on the other hand, $B$ has an additive cut, then two cases are possible: 
$|B|$ satisfies i.) $|B|\subset \cup\cap s_i$, and ii.) $|B|\supset \cap\cup s_i$.

i.) Since $|B|$ satisfies $|B|\subset \cup\cap s_i$, then there exists $i$ which satisfies $s_j\notin|B|$ for all $j>i$.
Hence there exists $i$ which satisfies $b_j<s_j$ for all $j>i$.
Since $B$ is additive, for each $k\in\FN$ there exists $i$ which satisfies $k\cdot b_j<s_j$ for all $j>i$.
It implies $\frac{b_j}{s_j}<\frac{1}{k}$ so that  $\varlimsup{\frac{F(B)}{(s_i)}} \subseteq {\frac{1}{\FN}}$ holds.

ii.) Since $|B|$ satisfies $|B|\supset \cap\cup s_i$, then there exists $i$ which satisfies $s_j\in |B|$ for all $j>i$.
Hence there exists $i$ which satisfies $b_j>s_j$ for all $j>i$.
Since $B$ is additive, for each $k\in\FN$ there exists $i$ which satisfies $b_j>k\cdot s_j$ for all $j>i$.
It implies $\frac{b_j}{s_j}>k$, so that  $\varliminf{\frac{F(B)}{(s_i)}}\supseteq {\frac{\FN}{1}}$ holds.
\end{proof}

Moreover, ratios of cuts of $B$ and $C$ to $S$ are equal as long as their lower cuts are equal.
The next theorem states this property, which is essentially the same as Theorem 3.4 of Kalina \cite{Kalina1989-a}.

\begin{thm}\label{uniqueRC}
  Let Borel semisets $B,C\in\mathscr{B}_s$ satisfy $\underline{B}\doteq
  \underline{C}$.
  Then, it is satisfied that $\varliminf{\frac{F(B)}{(s_i)}}\doteq
  \varliminf{\frac{F(C)}{(s_i)}}
  $ for any given BAF $F$ where $(s_i)_{i\in\FN}$ is an approximating sequence of a semiset $S$ having a nonadditive cut.
\end{thm}
\begin{proof}
  Given that both $B$ and $C$ have additive cuts $|B|=|C|$.
  Then, the equivalences follow by Theorem \ref{uniqueness}.
  
  Otherwise, both $B$ and $C$ have nonadditive cuts, so that there exists $d\in N$ which satisfies by Theorem \ref{KZ32} that
  \[
  \intrr{\mathscr{}}{d}\ \leq\
  \underline{B}\ \approx\
  \underline{C}\ \approx\
  \overline{B}\ \approx\
  \overline{C}\ \leq\
  \cl{\mathscr{}}{d}.
  \]
  Let $F(B)=(b_i)_{i\in\FN}$ and $F(C)=(c_i)_{\i\in\FN}$.
  Then, for each $k \in\FN\setminus 2$ there exists $i$ which satisfies the following inequations for all $j\in\FN\setminus i$
  \[
  \frac{d-\frac{d}{k}}{d}\ =\ 1-\frac{1}{k}\ <\
  \frac{b_j}{d}\ <\
  1+\frac{1}{k}\ =\ \frac{d+\frac{d}{k}}{d}  
  \]
  and
  \[
  \frac{d-\frac{d}{k}}{d}\ =\ 1-\frac{1}{k}\ <\
  \frac{c_j}{d}\ <\
  1+\frac{1}{k}\ =\ \frac{d+\frac{d}{k}}{d}.  
  \]
  It implies that there exists $j\in\FN$ which satisfies $\left| 1-\frac{b_j}{c_j}\right|<|\ell|$ for any given $\ell\in\FQ$, so that $\varliminf_{i\in\FN}{\frac{b_i}{c_i}}\doteq
  {1}$
  and $\varliminf_{i\in\FN}{\frac{b_i}{s_i}}\doteq
  \varliminf{\frac{c_i}{s_i}}$ holds.
  Consequently, $\varliminf{\frac{F(B)}{(s_j)}}\doteq
  \varliminf{\frac{F(C)}{(s_j)}}
  $ follows.
\end{proof}

As we shall see, this theorem assures us that Kalina measures of the Borel semisets which have nonadditive cuts coincide with classical measures when $S$ has nonadditive cut.
However, the next theorem, which is essentially the same as Theorem 4.2 of Kalina \cite{Kalina1989-a}, shows us somewhat strange aspects of rational cuts.
\begin{thm}\label{upper}
Let $S$ and $(s_i)_{i\in\FN}$ be the same as the last theorem except that $S$ has an additive cut.
Let $B\in\mathscr{B}_s$ and $|B|=\cup\cap(s_i)_{i\in\FN}$.
Then, 
\begin{enumerate}[i.)]
\item for each positive $r\in\FQ$ there exists $F$ which satisfies ${\frac{F(B)}{(s_i)}}={r}$. 
Moreover there exist $(\underline{b_i})_{i\in\FN}$ and $(\overline{b_i})_{i\in\FN}$ which satisfies
\[
\varliminf{\frac{\underline{b_i}}{s_i}}={\frac{1}{\FN}}\quad \text{ and }\quad 
\varliminf{\frac{\overline{b_i}}{s_i}}={\frac{\FN}{1}},
\]

\item and also exists $G$ in which $\varliminf{\frac{G(B)}{(s_i)}}\not\doteq
  \varlimsup{\frac{G(B)}{(s_i)}}$.
\end{enumerate}
\end{thm}
\begin{proof}
  To simplify the proof, let us restrict our attention to the case where $|B|$ is a $\sigma$-class.
  The other case where $|B|$ is $\pi$-class can be proven similarly.

  Let us start with a proof of i.) by showing that for each positive $r\in\FQ$ there exists $F$ which satisfies $\varliminf{\frac{F(B)}{(s_i)}}={r}$.
To build such $F$, put 
\[
F(B)=(b_i)\quad \text{in which}\quad b_i\ =\ \lfloor r\cdot s_i\rfloor.
\]
Since $B$ is additive, the equation $\cap\cup(b_i)\ =\ \cap\cup(s_i)\ =\ \cup\cap(s_i)\ =\ \cup\cap(b_i)\ =\ |B|$ holds.
Thus, $(b_i)$ is also an approximating sequence of $B$.

To construct a BAF $(\underline{b_i})_{i\in\FN}$ which satisfies $\varliminf{\frac{\underline{b_i}}{s_i}}={\frac{1}{\FN}}$, let us define a sequence of natural numbers $(n(k))_{k\in\FN\setminus\{0\}}$ inductively as
\[
n(1)\ =\ 0\quad\text{and}\quad n(k)\ =\ \min \left\{ i\in\FN\ ;\ \left(i>n(k-1)\right)\wedge \left(\forall j\geq i\right)\left(\left\lfloor\frac{s_j}{k}\right\rfloor\geq s_k\right)\right\}.
\]
Since $|B|$ is $\sigma$-additive, such $n(k)$ exists for each $k$.
Then put\footnote{
Simply putting
$\underline{b_i} \ =\ \left\lfloor \frac{s_i}{i}\right\rfloor$
may seem to work, but it doesn't.
Since it may happen that there exists $k\in\FN$ satisfying $\underline{b_i} \leq k$ for all $i\in\FN$ so that $\cap\cup \underline{b_i}<|B|$.
}
\[
\underline{b_i}\ =\ \left\lfloor\frac{s_i}{k}\right\rfloor \quad\text{ for }\quad n(k)\leq i< n(k+1).
\]
It is obvious that $(\underline{b_i})_{i\in\FN}$ approximates $B$.
Since $k$ increases limitlessly as $i$ grows, the equations $\varliminf{\frac{\underline{b_i}}{s_i}}=\varliminf{\frac{\underline{1}}{k}}={\frac{1}{\FN}}$ follow.

Lastly, to build a BAF $(\overline{b_i})_{i\in\FN}$ which satisfies $\varliminf{\frac{(\overline{b_i})}{(s_i)}}={\frac{\FN}{1}}$, put simply
\[
\overline{b_i} \ =\ i\cdot s_i.
\]
Then $(\overline{b_i})_{i\in\FN}$ approximates $B$ and, thus, $\varliminf{\frac{(\overline{b_i})}{(s_i)}}={\frac{\FN}{1}}$ follows.
\medskip

To prove ii.) let us consider two different BAFs $F$ and $H$ such that $\lim{\frac{F(B)}{(s_i)}}\not\doteq
\lim{\frac{H(B)}{(s_i)}}$.
Let also $F$ and $H$ satisfy $F(B)=(o_i)_{i\in\FN}$ and $H(B)=(e_i)_{i\in\FN}$.
Define another BAF $G$ as
\[
G(B)\ =\ \begin{cases}
o_i & \quad \text{if $i$ is odd}\\
e_i & \quad \text{otherwise}
\end{cases}
\]
Then, obviously $G(B)$ approximates $B$ but $\varliminf{\frac{G(B)}{(s_i)}}\not\doteq
\varlimsup{\frac{G(B)}{(s_i)}}$.
\end{proof}

The theorem shows that the ratio of $B$ to $S$ both of which have the same additive cut can be any rational cuts ranging from ${\frac{1}{\FN}}$ to ${\frac{\FN}{1}}$, which correspond to the results of external and internal quotient, respectively.
This fact indicates that the values of division by additive cuts are determined by the way they are approximated.
$B$ may have different ratio to $S$ depending on the BAF, while external and internal quotients remain constant at its minimum $\frac{1}{\FN}$ and its maximum $\frac{\FN}{1}$.

Remaining two cases of ratio of cut $B$ having an additive cut to $S$ are given by the next lemma, which is essentially the same as Lemma 4.1 of Kalina \cite{Kalina1989-a}.

\begin{lemma}
Let $S$ and $(s_i)_{i\in\FN}$ be the same as the last theorem.
Let $B\in\mathscr{B}_s$, $(b_i)_{i\in\FN}$ be an approximating sequence of $B$, and $F$ be any BAF. 
Then the following inclusions hold:
\begin{enumerate}[i.)]
\item if $\underline{B}\supset\cap\cup(s_i)_{i\in\FN}$ then $\varliminf{\frac{F(B)}{(s_i)}}\supseteq {\frac{\FN}{1}}$,
\item if $\overline{B}\subset\cup\cap(s_i)_{i\in\FN}$ then $\varlimsup{\frac{F(B)}{(s_i)}}\subseteq {\frac{1}{\FN}}$.
\end{enumerate}
\end{lemma}

\begin{proof}
To prove ii.), let $F(B)=(b_i)_{i\in\FN}$ and $d\in\cup\cap(s_i)_{i\in\FN}\setminus \overline{B}$.
Then there exists $i$ which satisfies $s_j<d$ for all $j>i$.
Since $|S|$ is additive and $d\in|S|$, for each $k$ there exists $i$ such that $s_j>d\cdot k$, so that $\frac{b_j}{s_j}<\frac{d}{s_j}<\frac{1}{k}$ holds for all $j>i$, hence $\varlimsup{\frac{F(B)}{(s_i)}}\subseteq{\frac{1}{\FN}}$ is satisfied.

The proof of i.) is essentially the same.
\end{proof}

Notice that all the theorems so far deal solely with Borel semisets but nonsemiset.
It is because the set-theoretically definable proper classes, such as the universal class $V$ or real intervals, have any approximating sequences since, say, $|V|=|[0,1]|=\cup_{i\in N}\isrc{i}$.
This problem is dealt in the next section.

\section{$\sigma$-consistency of Approximating Sequences}
Additive cuts face some difficulty also when adding them countably many times.
By definition given in the last of Section \ref{cuts}, the lower bound of countable sum of $\frac{1}{\FN}$ is $\frac{1}{\FN}$, while the upper bound is $\frac{\FN}{1}$.
However, it is also possible to set value for each countable sum by utilizing their approximating sequences.
To see how it is possible, let us start with building $\sigma$-partitions.

$\mathscr{S}$ as defined below is said to be a \textit{$\sigma$-partition} of $S$
\[
\mathscr{S}\ =\ \left\{S_{(i,j)}\subseteq S\ ;\
\begin{matrix}
  (\forall k\in\FN)(\forall m\in j(k))
  \left(
  \begin{matrix}
    \left(
    \left(i\ne k \right)\vee\left(j\ne m\right)
    \right)\\
    \Rightarrow
    \left(S_{(i,j)}\cap S_{(k,m)}=\emptyset\right)
  \end{matrix}
  \right)\\
  \wedge\left(\cup_{i\in\FN}\cup_{j\in j(i)} S_{(i,j)}=S\right)
\end{matrix}
\right\} \quad
\]
where $j(\cdot):\alpha\rightarrow\beta$ for some $\alpha,\beta\in N\setminus\FN$ which satisfies $j(i)\in\FN$ for each $i\in\FN$.
In particular, when $j(i)=1$ for all $i\in\FN$ the condition is simplified as follows
\[
\mathscr{S}\ =\ \left\{S_{i}\subseteq S\ ;\
  \left((\forall j\ne i)\left(S_{i}\cap S_{j}=\emptyset\right)\right)
  \wedge\left(\cup_{i\in\FN} S_{i}=S\right)
\right\}.
\]
Notice that if $j(\cdot)$ is, say, the exponential, $\mathscr{S}$ is uncountable and, thus, can cover uncountable classes, while it is countable if $j(\cdot)$ is the polynomial.

To measure real intervals, assuring the existence of $\sigma$-partitions of real continua\footnote{
A continuum consists of a pair $\mathscr{C}=\langle C,\doteq_\mathscr{C}\rangle$, in which $C$ is set-theoretically definable and $\doteq_\mathscr{C}$ is an indiscernibility equivalence on $C$, and a monad of $\mathscr{C}$ is an equivalence class of $x\in C$ defined as $\mon{C}{x}=\{y\in C\ ;\ x\doteq y\}$.
See Sakahara and Sato \cite{topology} for details.
} is essential.
The next theorem assures us that every continuum has a $\sigma$-partition.

\begin{thm}\label{sigma-par-cont}
  Given a continuum $\mathscr{C}=\langle C,\doteq_\mathscr{C}\rangle$, the class of all monads is a $\sigma$-partition. 
\end{thm}

\begin{proof}
  Given that $(R_i)_{i\in\FN}$ is a decreasing generating sequence of $\doteq_\mathscr{C}$.
  Let $X_i$ be a maximal $R_i$-net, which does not include any distinct pair $x,y\in X_i$ such that $\langle x,y\rangle\in R_i$ and for each $x\in C$ there exists $y\in X_i$ satisfying $\langle x, y\rangle\in R_i$.
  Let $(X_i)_{i\in\FN}$ be a monotone sequence of maximal $R_i$-nets and define $C_0=X_0$ and $C_i=X_i\setminus X_{i-1}$ for each $i>0$.

  Since $|C_i|\in\FN$ by the second theorem at the p.85 of Vop\v{e}nka \cite{ast}, a function $j(\cdot)$ defined as $j(i)=|C_i|$ for each $i\in\FN$ is finite valued, so that each member of $C_i$ for each $i\in\FN$ can be numbered with $j\in j(i)$ as $C_i=\cup_{j\in j(i)} \{c_{(i,j)}\}$.
  Since $X_i$ is maximal $R_i$-net for each $i$, $R_i``X_i=C$ and each pair $x,y\in X_i$ is mutually discernible.
  It implies that  $\mathscr{S}_X=\{\mon{C}{c_{(i,j)}}\ ;\ (i\in\FN)(j\in j(i)) \}$ forms a $\sigma$-partition.
\end{proof}
The class $\mathscr{S}_X$ constructed in the last proof is said to be a \textit{topological $\sigma$-partition} numbered according to $(X_i)_{i\in\FN}$, and the class $\choice{X}{\mathscr{C}}$ defined as 
$\choice{X}{\mathscr{C}}=\{c_{(i,j)}\,;\, (\exists i\in\FN)(\exists j\in j(i))(c_{(i,j)}\in C_i)\}$
is said to be a \textit{choice class} of $\mathscr{C}$ according to $(X_i)_{i\in\FN}$. 
By definition, $\choice{X}{\mathscr{C}}$ is a $\sigma$-class.

A BAF $F$ is said to be \textit{$\sigma$-consistent on} $\mathscr{S}$ iff for all $i\in I$ and $j\in h(i)$ where $I\subseteq \FN$ and $h(i)\subseteq j(i)$ the following equation holds:
\[
\varliminf{\frac{\sum_{i\in I}\sum_{j\in h(i)}F(S_{(i,j)})}{F(S)}}\ \doteq
\ \varliminf{\frac{F(\cup_{i\in I}\cup_{j\in h(i)}S_{(i,j)})}{F(S)}},
\]
where the sum $\sum_{i\in I}F(S_{(i,j)})$ of approximating sequences $F(S_{(i,j)})=(s_{(i,j),k})_{k\in\FN}$ abbreviates the sequence $\left(\sum_{i\in I}s_{(i,j),k}\right)_{k\in\FN}$.

Each $\sigma$-class has one trivial $\sigma$-consistent BAF.
Firstly, let us define, for notational ease, a \textit{$\ell$-step lagged approximate sequence} of $\alpha$, denoted as $(\alpha)_\ell$, which consists of $a_i=0$ for $i<\ell$ and $a_i=\alpha$ for $i\geq \ell$.
When $\ell=0$, $(\alpha)_0$ is  denoted simply as $(\alpha)$ and said to be a \textit{constant approximate sequence} of $\alpha$.

Then, the BAF on the $\sigma$-class $\mathrm{Ch}=\{c_{(i,j)}\,;\,(\forall i\in\FN)(\exists j\in j(i))\}$, which is said to be a \textit{canonical} BAF of $\mathrm{Ch}$, is given as
\[
F_{\mathrm{Ch}}(\{c_{(i,j)}\})\ =\ \left(1\right)_i\quad\text{for each $i\in\FN$ and $j\in j(i)$}
\]
and $F_{\mathrm{Ch}}(B)\ =\ \sum_{c\in B} F_{\mathrm{Ch}}(\{c\})$ for each $B\subseteq \mathrm{Ch}$.

While every canonical BAF is $\sigma$-consistent by definition, BAF in general is not.
However, as shown in the next theorem, every continuum has one.

\begin{thm}\label{sigmacnt}
  Let $\mathscr{C}$ be a continuum.
  Then, there exists a $\sigma$-consistent BAF on any topological $\sigma$-partition $\mathscr{S}_X$ of $\mathscr{C}$.
\end{thm}

\begin{proof}
  Let $(R_i)_{i\in\FN}$ be a monotone generating sequence of $\doteq_\mathscr{C}$ and $(X_i)_{i\in\FN}$ be a sequence of maximal $R_i$-nets.
  Define for each $c_{(i,j)}\in\choice{X}{\mathscr{C}}$ a sequence $\left(S_{(i,j),k}\right)_{k\in\FN}$
  as in Theorem \ref{countapp} as
  \[
  S_{(i,j),k}\ =\
  \begin{cases}
    \emptyset & \text{ if } \quad k<i\\
    R_k``\{c_{(i,j)}\} & \text{ otherwise}.
  \end{cases}
  \]
  Then, each sequence  $\left(S_{(i,j),k}\right)_{k\in\FN\setminus i}$ is monotone and generates $\mon{C}{c_{(i,j)}}$, so that $(|S_{(i,j),k}|)_{k\in\FN}$ approximates  $\mon{C}{c_{(i,j)}}$ by Theorem \ref{monotone-BGF}.

  Given $x,y\in \choice{X}{\mathscr{C}}$, let $(S_{x,k})_{k\in \FN}$ and $(S_{y,k})_{k\in \FN}$ denote their generating sequences.
  Then, there exists $k\in\FN$ which satisfies the following equation for all $\ell>k$,
  \[
  S_{x,\ell}\cap S_{y,\ell}\ =\ \emptyset.
  \]
  It implies that $\cup\cap|S_{x,\ell}\cap S_{y,\ell}|=0$, so that the following equation holds:
  \[
  \cup\cap |S_{x,\ell}\cup S_{y,\ell}| \ =\ |\mon{C}{x}|+|\mon{C}{y}|.
  \]

  Let $B$ be a subclass of $\choice{X}{\mathscr{C}}$ which satisfies $\cup_{i\in I}\cup_{j\in h(i)}\{c_{(i,j)}\}=B$ for some $I\subseteq \FN$ and $h(\cdot)\subseteq j(\cdot)$.
  Then, $\left(\cup_{\ell\in i}\cup_{j\in h(\ell)} S_{(i,j),k}\right)_{k\in\FN}$  generates $\fig{C}{\cup_{\ell\in i}\cup_{j\in h(\ell)} S_{(i,j),k}}$ for all $i\subseteq I$ by Lemma \ref{fingap}, and its cut approximates the figure too by Theorem \ref{monotone-BGF}.
  
  Now, define $F_{\mathscr{S}_X}(\fig{\mathscr{C}}{B})$ as
  \[
  F_{\mathscr{S}_X}\left(\fig{C}{B}\right)
  \ =\ \Biggl(\Bigl|\cup_{\ell\in I}\cup_{j\in h(\ell)} S_{(i,j),k}\Bigr|\Biggr)_{k\in\FN}.
  \]
  Then,  $F_{\mathscr{S}_X}$ is $\sigma$-consistent on $\mathscr{S}_X$.
\end{proof}

In general, a BAF $F_{\mathscr{S}_X}$ is said to be \textit{topological} iff it gives each monad $\mon{C}{c_{(i,j)}}$ a series $\left(|S_{(i,j),k}|\right)_{k\in\FN}$ where $S_{(i,j),k}=\emptyset$ if $k<i$ and, otherwise, $S_{(i,j),k}=R_k``\{c_{(i,j),k}\}$. 
Moreover, $F_{\mathscr{S}_X}$ is said to be \textit{almost uniform} on $\mathscr{C}$ iff there exists a finite set $b\subseteq \choice{X}{\mathscr{C}}$ which satisfy $|S_{x,k}|=n_k\in N$ for all $x\in \choice{X}{\mathscr{C}}\setminus b$ and $k\in\FN$, and also satisfy $p\leq\frac{|S_{x,k}|}{|S_{y,k}|}\leq q$ for all $x,y\in \choice{X}{\mathscr{C}}$ for some $p,q\in\FQ$.
If $b$ is empty, $(R_k)_{k\in\FN}$ is simply said to be \textit{uniform}.

\begin{thm}\label{coincidence}
  Let $F_{\mathscr{S}_X}$ be an almost uniform topological $\sigma$-consistent BAF on a topological $\sigma$-partition $\mathscr{S}_X$ of a continuum $\mathscr{C}$, and
  $F_{\choice{X}{\mathscr{C}}}$ be a canonical BAF on $\choice{X}{\mathscr{C}}$.
  Then, $\varliminf{\frac{F_{\choice{X}{\mathscr{C}}}({B})}{{F_{\choice{X}{\mathscr{C}}}}(\choice{X}{\mathscr{C}})}}\doteq
  \varliminf{\frac{F_{\mathscr{S}_X}(\fig{C}{B})}{F_{\mathscr{S}_X}(C)}}$ holds for all $B\subseteq \choice{X}{\mathscr{C}}$.
\end{thm}

\begin{proof}
  Since $F_{\mathscr{S}_X}$ is topological and  almost uniform, there exists  $p,q\in\FQ$, where $p<1<q$, satisfy
  \[
  p\cdot n_k\ \leq |R_k``\{c_{(i,j),k}\}|\ \leq\ q\cdot n_k. 
  \]
  Moreover, since $F_{\mathscr{S}_X}$ is almost uniform, there exists finite subset $b\subseteq B$ which satisfies for some $\ell\in\FN$ that 
  \[
  \left(\sum_{i\in\ell\cap I} h(i)+(p-1)|b|\right)n_k
  \ \leq\ 
      {\sum_{i\in\ell\cap I}\sum_{j\in h(i)} |S_{(i,j),k}|}
  \ \leq\ 
  \left(\sum_{i\in\ell\cap I} h(i)+(q-1)|b|\right)n_k,
  \]
  where $I\subseteq \FN$ and $h(\cdot)\subseteq j(\cdot)$ satisfying $\cup_{i\in I}\cup_{j\in h(i)}\{c_{(i,j)}\}=B$.
  
  Since for each $r\in\FQ$ there exists $k>\ell$ which satisfies
  \[
  \frac{(1-p)|b|}{\sum_{i\in k} j(i)}<r\quad\text{ and }\quad\frac{(q-1)|b|}{\sum_{i\in k} j(i)}<r,
  \]
  the following inequations hold for all $r\in\FQ$
  \[
  \lim_{\ell\in\FN}\frac{\sum_{i\in\ell\cap I} h(i)}{\sum_{i\in\ell} j(i)}-r \leq
  \lim_{\ell\in\FN}\frac{\sum_{i\in\ell\cap I}\sum_{k\in h(i)} |S_{(i,j),k}|}{\sum_{i\in \ell}\sum_{k\in j(i)}|S_{(i,j),k}|}\ \leq\
  \lim_{\ell\in\FN}\frac{\sum_{i\in\ell\cap I} h(i)}{\sum_{i\in\ell} j(i)}+r.
  \]
  It implies that $\varliminf {\frac{F_{\mathscr{S}_X}(\fig{\mathscr{C}}{{B}})}{F_{\mathscr{S}_X}(C)}}\doteq
  \varliminf{\frac{F_{\choice{X}{\mathscr{C}}}(B)}{F_{\choice{X}{\mathscr{C}}}(\choice{X}{\mathscr{C})}}}$ holds.
\end{proof}
 
Now, to see how rational additive cuts work, let us see some examples.
Consider that we have a unit-length string and cut it into two pieces with equal length countably many times.
What are left behind are countably many pieces, say, $\FN$, of strings\footnote{It may seem that there must be $2^\FN$ pieces left, and it is actually correct that $2^\FN$ cosists of uncountably many elements.
However, $2^\FN$ contains only countably many finite sets, that is, $\FN$.
What it includes beside these sets are proper semisets.
}, each of which has ${\frac{1}{\FN}}$ length.

Let $\langle[0,1]_{2^\varepsilon},\doteq\rangle$ be a continuum representing a real interval $[0,1]$, where $[0,1]_{2^\varepsilon}$ is a rational grid defined as $\left\{\frac{n}{2^\varepsilon}\,|\, (0\leq n\leq 2^\varepsilon)\wedge(n\in N)\right\}$ for some $\varepsilon\in N\setminus \FN$, and $\doteq$ is an indiscernibility equivalence defined by a generating sequence  $R=(R_k)_{k\in\FN}$ consisting of $R_k=\{\langle a,b\rangle\in[0,1]_{2^\varepsilon}^2\,;\,|a-b|\leq 2^{-k}\}$\footnote{See Sakahara and Sato \cite{real} for details}.
Moreover, a rational number $\alpha\in[0,1]_{2^\varepsilon}$ is said to be a \textit{position}, while a real $\mon{}{\alpha}\in[0,1]$ containing a position $\alpha$, is said to be a \textit{point} on a continuum separately.

\begin{ex}\label{continua}
  Let $c_{(i,j)}$ denote a position in a rational grid $[0,1]_{2^\varepsilon}$ where $i\in\FN$, $j\in j(i)$ in which $j(0)=j(1)=1$ and  $j(i)=2^{i-2}$ for all $i\geq 2$, as
  \[
  c_{(i,j)}\ =\
  \begin{cases}
    i & \text{ if } i<2\\
    \frac{2j+1}{2^{i-1}} & \text{ otherwise.}
  \end{cases}
  \]
  Then, the class consisting of all finite rational points constitutes a choice class of $[0,1]$
  \[
  \choice{X}{[0,1]}=\{c_{(i,j)}\,;\, \left((i<2)\wedge(j=0)\right)\vee(\exists i\in\FN\setminus 2)(j\in 2^{i-2})\},
  \]
  and $\mathscr{S}_X=\{\mon{}{c}\,;\, c\in \choice{X}{[0,1]}\}$ turns to be a topological $\sigma$-partition of $[0,1]_{2^\varepsilon}$ since $\mon{}{c}\cap\mon{}{d}=\emptyset$ for all $c\not\doteq d$ and $[0,1]_{2^\varepsilon}=\cup_{c\in\choice{X}{[0,1]}}\mon{}{c}$.

  Let $F_{\mathscr{S}_X}$ be a topological BAF on $\mathscr{S}_X$ where $S_{(i,j),k}=R_k``\{c_{(i,j),k}\}$ for all $k\geq i$ and, otherwise, $S_{(i,j),k}=\emptyset$. 
  Then, the value for each element of $\mathscr{S}_X$ is given as $F_{\mathscr{S}_X}(\mon{}{c_{(i,j)}})\ =\ (|S_{(i,j),k}|)_{k\in\FN}$ where
  \[
  |S_{(i,j),k}|\ =\ 2^{\varepsilon-k+1-[i<2]\footnotemark}\cdot [k\geq i]
  \]
  \footnotetext{$[i<2]$ is an Iverson bracket meaning that it equals 1 if $i<2$ and 0 otherwise.}
  and the rational cut of each point is given as
  \[
  \lim {\frac{F_{\mathscr{S}_X}\bigl(\mon{}{\{c_{(i,j)}\}}\bigr)}{F_{\mathscr{S}_X}([0,1])}}
  \ =\  \lim{\frac{(2^{\varepsilon-k+1-[i<2]}\cdot [k\geq i])_{k\in\FN}}{(2^\varepsilon)}} 
  \ =\ {\frac{1}{\FN}}.
  \]
  The rational cut of the interval $[0,\ell]$, where $2^\varepsilon\cdot \ell\in 2^\varepsilon$, is also given as 
  \[
  \lim{\frac{\sum_{i\in\FN}\sum_{j\leq 2^i\cdot \ell}F_{\mathscr{S}_X}(\mon{}{\{c_{(i,j)}\}})}{F_{\mathscr{S}_X}([0,1])}}
  \ =\ \lim{\frac{\left((2m(k)-1)\cdot 2^{\varepsilon-k}\right)_{k\in\FN}}{(2^\varepsilon)}}
  \ \doteq
  \ {\ell},
  \]
  where $m(k)=\max_{j\in 2^{k-1}}\{j\,;\,\frac{j}{2^{k-1}}\leq \ell\}$.

  Since each element satisfies $|S_{(i,j),k}|=2^{\varepsilon-k+1}$ unless $i<2$, $F_{\mathscr{S}_X}$ is almost uniform, so that the same rational cuts can be drawn directly from the choice class $\choice{X}{[0,1]}$ by Theorem \ref{coincidence}.
  Let $F_{\choice{X}{[0,1]}}$ be a canonical BAF of $\choice{X}{[0,1]}$.
  Then, since its element equals to that of $F_{\mathscr{S}_X}$ divided by $2^{\varepsilon-k}$, the value of $\choice{X}{[0,1]}$ is given as
  \[
  F_{\choice{X}{[0,1]}}(\choice{X}{[0,1]})\ =\
  \sum_{c\in\choice{X}{[0,1]}}F_{\choice{X}{[0,1]}}(c)\ =\ (2^k)_{k\in\FN}. 
  \]
  Similarly, the rational cut of each point is given as
  \[
  \lim {\frac{F_{\choice{X}{[0,1]}}(c_{(i,j)})}{F_{\choice{X}{[0,1]}}([0,1])}}\ =\ \lim{\frac{([k\geq i])_{k\in\FN}}{(2^k)_{k\in\FN}}}\ =\ {\frac{1}{\FN}},
  \]
  and the rational cut of the interval $[0,\ell]$ as
  \[
  \lim{\frac{\sum_{i\in \FN}\sum_{j\leq 2^i\cdot \ell}F_{\choice{X}{[0,1]}}(c_{(i,j)})}{F_{\choice{X}{[0,1]}}([0,1])}}
  \ =\ \lim{\frac{(m(k))_{k\in\FN}}{(2^k)_{k\in\FN}}}
  \ \doteq
  \ {\ell}.
  \]
\end{ex}

\begin{ex}\label{anothercontinua}
  The $\sigma$-partition of real unit interval can be enumerated differently from Example \ref{continua}.
  To see an example, let us consider another generating sequence $(P_{i})_{i\in\FN}$  of $\doteq$ as follow:
  \[
  P_k\ =\ \left\{\langle a,b\rangle\in [0,1]_{2^\varepsilon}^2\,;\,
  \begin{pmatrix}
    \left(\left(\max\{a,b\}\leq\frac{3}{4}\right)\Rightarrow \left(\frac{2}{3}\cdot|b-a|\leq\frac{1}{2^{k}}\right)\right)\\
    \wedge\left(\left(\frac{3}{4}\leq \min\{a, b\}\right)\Rightarrow\left({2}\cdot|b-a|\leq\frac{1}{2^{k}}\right)\right)\\
    \wedge
    \begin{pmatrix}
    \left(\min\{a,b\}\leq\frac{3}{4}\leq \max\{a,b\}\right)\\
    \Rightarrow\left( 2\max\{a,b\}-1-\frac{2}{3}\min\{a,b\}\leq \frac{1}{2^{k}}\right)
    \end{pmatrix}\\
  \end{pmatrix}
  \right\}.
  \]
  It may seem complicated, however, $(P_i)_{i\in\FN}$ is simply connecting two generating sequences $(R_i)_{i\in\FN}$ and $(Q_i)_{i\in\FN}$ which are given as
  \[
  R_k\ =\ \left\{\begin{matrix}\langle a,b\rangle\in \left[0,\frac{3}{4}\right]_{2^\varepsilon}^2\,;\, \frac{2}{3}|b-a|\leq\frac{1}{2^{k}}
  \end{matrix}
  \right\}\ \text{ and }\ 
  Q_k\ =\ \left\{\begin{matrix}\langle a,b\rangle\in \left[\frac{3}{4},1\right]_{2^\varepsilon}^2\,;\, 2|b-a|\leq\frac{1}{2^{k}}\end{matrix}\right\}.
  \]
  It is easily confirmed that $(P_i)_{i\in\FN}$ is also a $\pi$-generating sequence\footnote{See Chapter 3 Section 1 of Vop\v{e}nka \cite{ast}}.
  To make sure, let us verify that $(P_i)_{i\in\FN}$ satisfies $\langle x,y\rangle,\langle y,z\rangle\in P_{k+1}$ then $\langle x,z\rangle\in P_k$.
  One of the nontrivial cases is $x\leq y\leq \frac{3}{4}\leq z$.
  Since $\langle x,y\rangle\in P_{k+1}$, the inequation $\frac{2}{3}(y-x)\leq\frac{1}{2^{k+1}}$ holds.
  The inequation $2z-1-\frac{2}{3}y\leq\frac{1}{2^{k+1}}$ is also satisfied since $\langle y,z\rangle\in P_{k+1}$.
  Adding these two inequation, the following inequation is met
  \[
  2z-1-\frac{2}{3}x\ \leq\ \frac{1}{2^{k+1}}+\frac{1}{2^{k+1}}\ =\ \frac{1}{2^{k}}.
  \]
  It implies that $\langle x,z\rangle\in P_{k}$.

  Another choice class $\choice{Y}{[0,1]}=\{{c_{(i,j)}}\,;\,(\exists i\in\FN)(j\in j(i))\}$ according to a sequence $(Y_k)_{k\in\FN}$ of maximal $P_k$-nets where $j(0)=j(1)=1$ and $j(i)=2^{i-2}$ is given as $c_{(i,0)}=i$ if $i<2$, otherwise,
  \[
  c_{(i,j)}= \begin{cases}
    \frac{3}{2}\cdot\frac{2j+1}{2^{i-1}} & \text{ if } j< 2^{i-3}\\
    \frac{1}{2}+\frac{1}{2}\cdot\frac{2j+1}{2^{i-1}} & \text{ if } j\geq 2^{i-3}
  \end{cases}
  \]
  and the topological $\sigma$-partition is given as $\mathscr{S}_Y=\{\mon{}{x}\,;\,x\in\choice{Y}{[0,1]}\}$.
  
  Let $F_{\mathscr{S}_Y}$ be a topological BAF where $S_{(i,j),k}=\emptyset$ if $k<i$ and, otherwise, $S_{(i,j),k}=P_k``\{c_{(i,j),k}\}$.
  Then, the length of the interval, say, $\left[0,\frac{3}{4}\right]$ is given as same as that of Example  \ref{continua} as
  \[
  \lim{\frac{F_{\mathscr{S}_Y}([0,\frac{3}{4}])}{F_{\mathscr{S}_Y}([0,1])}}
  \ =\ \lim{\frac{(\frac{3}{4}+\frac{1}{4}[k=0]+\frac{1}{2^{k+1}}[k\geq 2])_{k\in\FN}}{(1)}}
  \ \doteq
  \ {\frac{3}{4}}
  \]
  However, that of canonical BAF $\lim{\frac{F_{\choice{Y}{[0,1]}}([0,\frac{3}{4}])}{F_{\choice{Y}{[0,1]}}([0,1])}}={\frac{1}{2}}$.
\end{ex}

As one can see, the monads considered in two exmaples correspond to the Vitali set in ZFC settings.
It is well-known that the Vitali set is unmeasurable, however, it is possible to evaluate its ratio in AST settings.
In fact, its measure can also be enumerated as we will see in the next section.

\section{Kalina Measures}\label{kalina}

A \textit{Kalina measure}\footnote{Notice that the definition of Kalina measure is slightly different from that of Kalina \cite{Kalina1989-a} in which it is defined as
\[
m_{s,\,F}(B)\ =\ r\quad\text{iff}\quad (\exists \delta\in N\setminus\FN)(\delta<\alpha)(\forall \eta\in \delta\setminus\FN)\left(\frac{b_\eta}{s_\eta}\in r\right),
\]
where approximating sequences are defined to be consisting of $\delta\in N\setminus\FN$ elements, so that $\frac{b_\eta}{s_\eta}$ exists.
It is easy to verify that, when $(a_i)_{i\in\FN}$ converges to $x$, there exists a prolonged sequence $(a_i)_{i\in\delta}$ which satisfies $a_\eta\in\mon{}{x}$ for all $\eta\in\delta\setminus\FN$ by the axiom of prolongation.
} $m_{s,\,F}\ ;\  \mathscr{O}(s,F)\rightarrow R\cup\{\infty\}$ is defined as
\[
m_{s,\,F}(B)\ =\ \lim_{i\in\FN}\mon{}{\textstyle \frac{b_i}{s_i}}.
\]
where $s=(s_i)_{i\in\FN}$ is an approximating sequence of $S$ and $F(B)=(b_i)_{i\in\FN}$.
A real sequence $(\mon{}{a_i})_{i\in \FN}$ is said to \textit{converge\footnote{
  See Sakahara and Sato \cite{topology} for a general statement.} to} $\mon{}{x}$, denoted as $\lim_{i\in\FN} \mon{}{a_i}= \mon{}{x}$, iff
\[
\left(
\forall q\in\FQ
\right)
\left(
\exists i\in\FN
\right)
\left(
\forall j>i
\right)
\begin{pmatrix}
\left(
| a_j - x|< |q|
\right)\\
\vee
\left(
a_j>q\wedge x>q
\right)\\
\vee
\left(
a_j<q\wedge x<q
\right)
\end{pmatrix}
.
\]
When $F_{\mathscr{S}}$ is $\sigma$-consistent on $\mathscr{S}$, $m_{F_\mathscr{S}(S),F_\mathscr{S}}(B)$ is abbreviated as $m_{F_\mathscr{S}}(B)$, and $\langle F(S),F\rangle$-observable class is simply called $F_\mathscr{S}$-observable.

It is easy to confirm that a Kalina measure of any given Borel semiset $B$ corresponds in a certain manner to the rational cuts given by $B$.
The next proposition shows how.
\begin{prop}
  A Kalina measure $m_{s,F}(B)$ satisfies the following equation for each $B\in\mathscr{B}_s$
  \[
  m_{s,F}(B)\ =\
  \begin{cases}
    \mon{}{q} & \text{ if }\ \lim{\frac{F(B)}{s}}\ \doteq
    \ {q}\\
    \mon{}{0} & \text{ if }\ \varlimsup{\frac{F(B)}{s}}\ \subseteq\ {\frac{1}{\FN}}\\
    \infty & \text{ if }\ \varliminf{\frac{F(B)}{s}}\ \supseteq\ {\frac{\FN}{1}}\\
  \end{cases}
  \]
\end{prop}

\begin{proof}
  Let us start with the case where $\lim{\frac{F(B)}{s}}\ \doteq
  \ {q}$ holds.
  It implies that the inequation $\lim |\frac{b_i}{s_i}-q|\leq |a|$ holds for all $a\in\FQ$,
  thus, there exists $i\in\FN$ which satisfies $|\frac{b_j}{s_j}-q|\leq |a|$ for each $j>i$ and $a\in\FQ$,
  so that  $\lim \mon{}{\frac{b_i}{s_i}}=\mon{}{q}$ holds.

  Secondly, consider the case $\varlimsup{\frac{F(B)}{s}}\ \subseteq\ {\frac{1}{\FN}}$ holds.
  It implies that for each $a\in\FQ$, there exists $i\in\FN$ which satisfies ${0}\subseteq {\frac{b_j}{s_j}}\subseteq {|a|} $ for all $j>i$, so that $0\leq \frac{b_j}{s_j}\leq |a|$ holds for all $j>i$ and thus, $\lim\mon{}{\frac{b_i}{s_i}}=\mon{}{0}$.

  Lastly, consider the case $\varliminf{\frac{F(B)}{s}}\ \supseteq\ {\frac{\FN}{1}}$ holds.
  It implies that for each $a\in\FQ$, there exists $j>i$ which satisfies ${|a|} \subseteq {\frac{b_j}{s_j}}$ for all $i\in\FN$, so that $|a|\leq \frac{b_j}{s_j}$ holds for some $j>i$ and thus, $\lim\mon{}{\frac{b_i}{s_i}}=\infty$ follows.
\end{proof}

Let $s=(s_i)_{i\in\FN}$ be any approximating sequence of a Borel semiset having nonadditive cut, then, $\langle s,F\rangle$-observable class coincides with $\mathscr{B}_s$ by Theorem \ref{uniqueness}.
It implies the next theorem which corresponds to Theorem 3.7 of Kalina \cite{Kalina1989-a}.
\begin{thm}
Let $F$ be a BAF and $(s_i)_{i\in\FN}$ approximates a Borel semiset having nonadditive cut.
Then $\mathscr{O}(s,F)=\mathscr{B}_s$ and $m_{s,F}$ is nondecreasing and nonnegative measure.
\end{thm}

Notice that it is in stark difference with the original setting of Kalina \cite{Kalina1989-a} that $\sigma$-additivity of $F$ does not hold in our setting.
This disagreement stems from the way measures are defined.
Firstly, approximating sequences are defined in our setting as countable sequences, which enable us to capture characteristic rational cuts such as ${\frac{1}{\FN}}$, while Kalina \cite{Kalina1989-a} defined them as uncountable ones.
Secondly, prior to evaluating measures, rational cuts of classes are given in our setting, while Kalina \cite{Kalina1989-a} assigned real numbers directly to individual classes so that the real evaluation of ${\frac{1}{\FN}}$ cannot be differentiated from that of ${0}$.

Notice also that the Kalina measure of a lower cut $\underline{X}$ and an upper cut $\overline{X}$ of any given Borel semiset $X\in\mathscr{B}_s$ coincide if $s$ approximates nonadditive cut.
The next theorem, which follows directly by Theorem \ref{uniqueRC}, states this property.

\begin{thm}
  Let $B,C\in\mathscr{B}_s$ be such that $\underline{B}\doteq
  \underline{C}$. 
Then $m_{s,\,F}(B)=m_{s,\,F}(C)$.
\end{thm}

On the contrary, when $s=(s_i)_{i\in\FN}$ is an approximating sequence of that having additive cut, then every Borel semiset is not always $\langle s,F\rangle$-observable.
The next theorem states this, which is a direct consequence of Theorem \ref{upper}.
\begin{thm}
Let $B\in\mathscr{B}_s$, $|B|=\cup\cap (s_i)_{i\in\FN}$ be an additive cut.
Then
\begin{enumerate}[i{.)}]
\item for each real $0\leq r\leq \infty$ there exists a BAF $F$ such that $B\in\mathscr{O}(s,F)$ and $m_{s,F}(B)=r$
\item there exists a BAF $G$ such that $B\notin\mathscr{O}(s,G)$.
\end{enumerate}
\end{thm}

Moreover, when $S$ has an additive cut, the fundamental feature of Kalina measures emerges.
A BAF $F_\mathscr{S}$ is said to be \textit{$\sigma$-superadditive} when there exists a $\sigma$-partition $\mathscr{T}=\{S_{(i,j)}\subseteq T\ ;(i\in\FN)(j\in j(i))\}$ of a subclass $T$ of $S$ which satisfies $m_{\mathscr{S}}(S_{(i,j)})=\mon{}{0}$ for all $i\in\FN$ and $j\in j(i)$, and $m_{\mathscr{S}}(\cup_{i\in \FN}\cup_{j\in j(i)} S_{(i,j)})=\mon{}{r}\in(0,1]$.

\begin{thm}\label{sigma-con-BAF}
  Let $\mathscr{C}=\langle C,\doteq_\mathscr{C}\rangle$ be a continuum and $F_\mathscr{S}$ be an almost uniform topological BAF.
  Then, $F_\mathscr{S}$ is $\sigma$-superadditive.
\end{thm}

\begin{proof}
  Let $\mathscr{S}=\{S_{(i,j)}\ ;\ (i\in\FN)(j\in j(i))\}$ be an almost uniform topological $\sigma$-partition $\mathscr{C}$ and 
  $F_\mathscr{S}$ be a topological BAF given as
  \[
  F_\mathscr{S}(\cup_{i\in I}\cup_{j\in j(i)}S_{(i,j)})\ =\ \sum_{i\in I}\sum_{j\in j(i)}\left(| S(i)_{(i,j),k}|\right)_{k\in\FN}.
  \]
  Since $F_\mathscr{S}$ is almost uniform, there exists a finite set $b\subseteq \choice{R}{\mathscr{C}}$ which satisfies $|R_k``\{x_{(i,j),k}\}|=n_k$ for all $x_{(i,j),k}\in b$, and the following inequations hold
  \[
  |R_k``\{x_{(i,j),k}\}|\ \leq\ q\cdot n_k\quad\text{ where } \ \langle x_{(i,j),k},c_{(i,j)}\rangle\in R_k,
  \]
  and, thus,  $|S_{(i,j),k}|\leq q\cdot n_k$ for some $q\in\FQ$ for all $i\in\FN$.
  Consequently, the following inequation is met
  \[
  m_\mathscr{S}(S_{(i,j)})
  \ \leq\ \lim_{i\in\FN}\mon{}{\frac{q\cdot n_k}{\sum_{k\in i}j(k)\cdot n_k}}
  \ =\ \lim_{i\in\FN}\mon{}{\frac{q}{\sum_{k\in i}j(k)}}\ =\ \mon{}{0}.
  \]
  But ${{\sum_{i\in\FN}\sum_{j\in j(i)}F_\mathscr{S}(S_{(i,j)})}} = F_\mathscr{S}{(S)}$, so that $m_\mathscr{S}\left(\cup\cup S_{(i,j)}\right)= \mon{}{1}$ holds.
\end{proof}

Let us exhibit an examples of $\sigma$-partition and $\sigma$-consistent BAFs on  a unit interval $[0,1]$.
\medskip

\begin{ex}\label{kalex2}
  Let $[0,1]=\langle [0,1]_{2^\varepsilon},\doteq\rangle$ be a continuum of a closed unit interval examined in Example \ref{continua}.
  A Kalina measure of $[0,\ell]$ is given as
  \[
  m_{F_{\mathscr{S}_R}}([0,\ell])\ =\
  m_{F_{\choice{R}{[0,1]}}}\left(\cup_{i\in\FN}\cup_{j\leq 2^i\cdot \ell}\mon{}{c_{(i,j)}}\right)\ =\
  \mon{}{\ell}.
  \]
  However, a Kalina measure of each monad $\mon{}{q}$ is given as
  \[
  m_{F_{\mathscr{S}_R}}(\mon{}{q})\ =\
  m_{F_{\choice{R}{[0,1]}}}(q)\ =\ \mon{}{0}
  \]
  and, thus, $\sigma$-superadditivity follows.
\end{ex}

\section{Probability Measures and Random Variables}

Let $A$ be a nonempty Borel semiset and $F$ a Borel approximating function.
A function $X_F:A\rightarrow Q$ is said to be a \textit{random variable} with respect to $F$ if for each $L\in\mathscr{B}_Q$ it is satisfied that $(X_F^{-1})``L$ is $\langle F(A),F\rangle$-observable and the function $D:\mathscr{B}_Q\rightarrow R$ defined by
\[
D(L)=m_{F(A),F}\left(\left(X_F^{-1}\right)``L\right)
\]
is $\sigma$-superadditive.
The function $D$ is said to be the \textit{probability distribution} of $X_F$.

\begin{ex}
  One example of random variables on a uniform real interval $[0,1]$ examined in Example \ref{kalex2} is given simply as
\[
X_{F_{\mathscr{S}}}(x)\ =\ {x},
\]
and its probability distribution for any subclass $B\subseteq [0,1]$ is given simply as
\[
 D(B)\ =\ m_{F_{\mathscr{S}}}\left(\left(X_{F_{\mathscr{S}}}^{-1}\right)``B\right)\ =\ m_{F_{\mathscr{S}}}\left(B\right).
\]
\end{ex}

\section{Kalina Measurable Functions}

Let $\mathscr{C}=\langle C,\doteq_\mathscr{C}\rangle$ be a continuum and $\mathscr{R}=\langle Q,\doteq\rangle$ be a real continua.
A real function\footnote{The definition of real functions are given in 2.3 Morphism of Tsujishita \cite{tjst}. See also Section 3 of Sakahara and Sato \cite{real} for details.} $\mathscr{F}:\mathscr{C}\rightarrow \mathscr{R}$ is said to be \textit{measurable} if for each $r\in R$ which satisfies $r=\mon{}{q}$ for some $q\in Q$ there exist a Borel class $B\subseteq C$ and a $\mathscr{F}$'s representation $F$, that is, $\mathscr{F}=[F]$, such that 
\[
\left(\forall q\in Q\right)
\left(\exists B\in\mathscr{B}_C\right)
\left(B=\fig{}{\{x\in C;\, F(x)\leq q\}}\right),
\]
where $\mathscr{B}_C$ is the smallest $\sigma$-ring containing all the set-theoretically definable subclasses of $C$.
Then, the next theorem follows.
\begin{thm}
Every real function $\mathscr{F}$ is measurable.
\end{thm}
\begin{proof}
Since $F$ is a real class, $\{x\in C;\, F(x)\leq q\}$ is a real class too by the (b) of Lemma 1.2.1 of Kalina and Zlato\v{s} \cite{Kalina-Zlatos1988}.
Thus, its figure is an element of $\mathscr{B}_C$.
\end{proof}

\section{Integration}
Let $\mathscr{S}$ be a $\sigma$-partition of a continuum $\mathscr{C}=\langle C,\doteq_\mathscr{C}\rangle$. 
Let $\mathscr{F}:\mathscr{C}\rightarrow \mathscr{R}$ be a real function and $G_{\mathscr{S}}$ be a $\sigma$-consistent BAF which gives Borel approximating sequence $G_{\mathscr{S}}(F^{-1}(\{y\}))=(x_k)_{k\in\FN}$ for each $y\in F(B)$ and $G_{\mathscr{S}}(C)=(s_k)_{k\in\FN}$. 
Then, the \textit{integral of $\mathscr{F}$ over $\fig{}{B}$ with respect to $G_{\mathscr{S}}$} is given as
\[
  \int_{\mathscr{S}(B)} \mathscr{F}dm_{G_\mathscr{S}}\ =\ \lim_{k\in\FN}\mon{}{{{\sum_{y\in {F}(B)}y\cdot \frac{x_k}{s_k}}},}
\]
in which $\mathscr{S}(B)=\{S\in\mathscr{S}\,;\,S\subseteq \fig{}{B}\}$ and $F$ is a representation of $\mathscr{F}$.

\begin{thm}
  Let $[0,1]=\langle [0,1]_{2^\varepsilon},\doteq\rangle$ be a unit-length real continuum,  $\mathscr{S}$ be a $\sigma$-partition of $[0,1]$, and $G_\mathscr{S}$ be an almost uniform topological BAF.
  Then, the integration of $\mathscr{F}$ over $[0,\ell]$, where $\ell\in[0,1]_{2^\varepsilon}$, coincides with that of Riemann:
  \[
  \int_{[0,\ell]} \mathscr{F}\, dm_{G_\mathscr{S}}
  \ =\ \int_0^\ell \mathscr{F}(x) dx.\ \footnotemark
  \]
  \footnotetext{The definition is given as $\int_0^\ell \mathscr{F}(x)dx=\mon{}{\sum_{0\leq x\leq \ell} F(x)\Delta x}$ where $\Delta x=\frac{1}{2^\varepsilon}$.
    See Sakahara and Sato \cite{real} for detailed descriptions of this notation.}
\end{thm}

\begin{proof}
  Let us recall that the topological $\sigma$-partition $\mathscr{S}$ is given as
  \[
  \mathscr{S}\ =\ [0,1]_{2^\varepsilon}/\doteq.
  \]
  Let $G_{\mathscr{S}}$ be an almost uniform topological $\sigma$-consistent BAF.
  Then, the integral of $\mathscr{F}$ over $[0,\ell]=\fig{}{[0,\ell]_{2^\varepsilon}}$ is given as
  \[
  \int_{[0,\ell]}\mathscr{F}\, dm_{G_\mathscr{S}}\ =\ 
  \lim_{k\in\FN}\mon{}{\sum_{q\in F([0,\ell]_{2^\varepsilon})} q\cdot \frac{x_k(q)}{2^\varepsilon}}.
  \]
  where $(x_k(q))_{k\in\FN}=G_\mathscr{S}(F^{-1}(q))$.

  Since $G_\mathscr{S}$ is topological and  almost uniform, the following equation holds by Theorem \ref{coincidence}
  \[
    {\lim_{k\in\FN}\mon{}{\sum_{q\in F([0,\ell]_{2^\varepsilon})}q\cdot \frac{x_k}{2^k}}}\ =\
    \mon{}{\sum_{q\in F([0,\ell]_{2^\varepsilon})} q\cdot |F^{-1}(q)|\cdot\Delta x}.
  \]
  While the Riemann sum of the same interval is given as:
  \[
    \int_0^\ell \mathscr{F}(x) dx\ =\ 
    \mon{}{\sum_{0\leq x\leq \ell}F(x)\cdot\Delta x}\ =\
    \mon{}{\sum_{q\in F([0,\ell]_{2^\varepsilon})}q\cdot |F^{-1}(q)|\cdot\Delta x}.
  \]
  It implies that 
  \[
    \int_{[0,\ell]} \mathscr{F}\,dm_{G_\mathscr{S}}
  \ =\ \int_0^\ell \mathscr{F}(x) dx.
  \]
\end{proof}

\section{Concluding Remarks}
The present paper provides a way to measure any given continua including the ones which are nonmeasurable with the Lebesgue's type measurement. 
Two examples, which are dealt in Example \ref{continua} and \ref{anothercontinua}, are the classes corresponding to the Vitali set, which dissects a unit interval into a $\sigma$-partition. 
Another example, whose details are dealt in the accompanying paper \cite{nonamenables}, is the class resulting from Banach-Tarski's theorem \cite{Banach1924}, which divides a ball into a $\sigma$-partition consisting of choice classes from $G$-orbits. 
 
Both examples divide a class which has a finite positive measure to an equal sized $\sigma$-partition.
Since Lebesgue measures are required to preserve $\sigma$-additivity, each element having 0 measure cannot have positive measure even when countably many of them are added together.
In contrast, Kalina measures relax it to $\sigma$-superadditivity to make room for them to have positive measure. 

To be precise, it is enabled by the system of rational cuts which is extended from rational numbers by adding two types of additive cuts, $\isrc{\frac{1}{\FN}}$ and $\isrc{\frac{\FN}{1}}$. 
$\sigma$-addition of any finite rational numbers are set to equal to $\isrc{\frac{\FN}{1}}$, while any finite rational numbers divided by $\FN$ are, dually, determined to $\isrc{\frac{1}{\FN}}$.
The reason why $\sigma$-addition of $\isrc{\frac{1}{\FN}}$ can be any finite numbers is obtained by retracing this process in reverse, that is, any numbers divided into $\sigma$-partition can be restored to the original when they are added together.
These values are determined precisely by the generating sequence of each component class with the help of the axiom of choice.
The $\sigma$-superadditivity arises, superficially, only because the values of these measures are approximated by real numbers.

Since any real continua of AST are divided into $\sigma$-partition of points, Kalina measures on them  do not require null-classes in essence.
In another word, every class corresponding to a null-set of ZFC is all absorbed in a point of AST, so that $\sigma$-union of them covers each continuum with nothing left outside.

Notice that the axiom of choice, or that of extensional coding, is adopted here.
It is often considered to be the source of the problem of existence of some nonmeasurable classes.
However, as it is shown in the present paper, the axiom is indispensable to determine the exact value of each measure, and nothing to do with non-measurability.
The essence of the problem lies elsewhere.
The present paper offers one possible way to avoid the problem.

The way indefiniteness is introduced here is also worth mentioning.
AST succeeded in including indefiniteness into formal system of numbers as a form of semisets, or, specifically say, $\FN$.
The additive rational cut $\isrc{\frac{1}{\FN}}$ can be said to introduce indefiniteness in a dual way.
Making use of approximating sequences, the system provides a way to calculate with these indefinite numbers in an appropriate manner.
It may provide a whole new way to treat phenomena involving indefiniteness.

These kind of measurement matter typically in everyday situations.
When figuring out the ratio of the people who thinks the same way as one does, or that of the days something good will happen, the total number of people or  days remain indefinite.
We would conduct this kind of somewhat vague measuring in everyday life based on finite experiences and extending them to countably many possibilities, 
The framework developed here offers a way to model these kind of measuring in a proper manner.

\bibliographystyle{plain}
\bibliography{ref} 

\end{document}